\documentclass[12pt, reqno]{amsart}
\usepackage{eurosym}
\usepackage{amsmath, amsthm, amscd, amsfonts, amssymb, graphicx, color}
\usepackage[bookmarksnumbered, colorlinks, plainpages]{hyperref}

\hypersetup{colorlinks=true,linkcolor=red, anchorcolor=green, citecolor=cyan, urlcolor=red, filecolor=magenta, pdftoolbar=true}
\newtheorem{theorem}{Theorem}[section]
\newtheorem{lemma}[theorem]{Lemma}
\newtheorem{proposition}[theorem]{Proposition}
\newtheorem{corollary}[theorem]{Corollary}
\theoremstyle{definition}
\newtheorem{definition}[theorem]{Definition}
\newtheorem{example}[theorem]{Example}

\theoremstyle{remark}

\numberwithin{equation}{section}
\input{tcilatex}

\begin{document}
\title[Lip-Linear operators ]{Lip-Linear operators and their connection to
Lipschitz tensor products }
\author{Athmane Ferradi and Khalil Saadi}
\date{}
\dedicatory{Laboratoire d'Analyse Fonctionnelle et G\'{e}om\'{e}trie des
Espaces, Universit\'{e} de M'sila, Alg\'{e}rie.\\
Second address for the first author: Ecole Normale Sup\'{e}rieure de
Bousaada, Algeria\\
athmane.ferradi@univ-msila.dz\\
khalil.saadi@univ-msila.dz}

\begin{abstract}
The linear operators defined on the Lipschitz projective tensor product $X%
\widehat{\boxtimes }_{\pi }E$ motivate the study of a distinct class of
operators acting on the cartesian product $X\times E$. This class, denoted
by $LipL_{0}\left( X\times E;F\right) $, combines Lipschitz and linear
properties, forming an intermediate framework between bilinear and
two-Lipschitz operators. We establish an identification between this space
and $\mathcal{L}\left( X\widehat{\boxtimes }_{\pi }E;F\right) $, which also
connects it to the space of bilinear operators $\mathcal{B}\left( \text{$%
\mathcal{F}$}(X)\times E;F\right) $. Furthermore, we extend summability
concepts within this category, with a particular focus on integral and
dominated $(p;q)$-summing operators.
\end{abstract}

\maketitle

\setcounter{page}{1}


\let\thefootnote\relax\footnote{
\par
{}\textit{Key words and phrases. }Lip-Linear operators; Two-Lipschitz
operators; Lipschitz summing operators; integral Lip-Linear operators;
dominated $\left( p,q\right) $-summing Lip-Linear operators; Pietsch
factorization; factorization theorems.}

\section{Introduction and preliminaries}

Let $X$ be a pointed metric space, and let $E,F$ be Banach spaces. Inspired
by the idea presented in \cite[Proposition 6.15]{ccjv}, where the authors
defined mappings $T:X\times E\longrightarrow F$ that are Lipschitz for the
first component and linear for the second component, we investigate a
similar class of mappings defined on the cartesian product $X\times E$ with
values in a Banach space. The space of such mappings, denoted $LipL(X\times
E;F),$ is referred to as the space of Lip-Linear operators. In \cite{hadt},
the authors explore the concept of two-Lipschitz operators defined on the
cartesian product $X\times Y$ of two pointed metric spaces, where the
operators are Lipschitz with respect to both components.\textbf{\ }This
definition builds upon the work initiated by Dubei in 2009 \cite{d}. By
considering two-Lipschitz operators defined on $X\times E$ with the second
component being linear, we precisely recover the type of operators studied
in this paper. In the classical setting, tensor products serve as a powerful
tool to relate bilinear mappings to their corresponding linear mappings.
Specifically, the space of bilinear operators $\mathcal{B}\left( E\times
F;G\right) $ is isometrically identified with the space of linear operators $%
\mathcal{L}(E\widehat{\otimes }_{\pi }F;G),$ where $E\widehat{\otimes }_{\pi
}F$ is the projective tensor product of $E$ and $F$. A similar approach
applies to Lipschitz tensor products. By considering linear mappings defined
on the Lipschitz tensor product $X\boxtimes E,$ we aim to study how the
corresponding mappings behave on the cartesian product $X\times E.$ Let $%
u:X\boxtimes E\rightarrow F$ be a linear operator. The corresponding
operator $T_{u}:X\times E\rightarrow F$ defined by $T_{u}\left( x,e\right)
=u\left( \delta _{\left( x,0\right) }\boxtimes e\right) $ for each $x\in X$
and $e\in E,$ is Lip-Linear. Conversely, if $T:X\times E\rightarrow F$ is a
Lip-Linear operator, then there exists a linear operator $\widehat{T}%
:X\boxtimes E\rightarrow F$ such that%
\begin{equation*}
\widehat{T}(\dsum\limits_{i=1}^{n}\delta _{\left( x_{i},y_{i}\right)
}\boxtimes e_{i})=\dsum\limits_{i=1}^{n}T\left( x_{i},e_{i}\right) -T\left(
y_{i},e_{i}\right) .
\end{equation*}%
We therefore have the algebraic identification%
\begin{equation*}
LipL(X\times E;F)=\mathcal{L}(X\boxtimes E;F),
\end{equation*}%
which can be represented in the following diagram 
\begin{equation*}
\begin{array}{ccc}
X\times E & \overset{T}{\longrightarrow } & F \\ 
\sigma ^{LipL}\downarrow & \nearrow \widehat{T} &  \\ 
X\boxtimes E &  & 
\end{array}%
\end{equation*}%
where $\sigma ^{LipL}=\delta _{X}\boxtimes id_{E}.$ We equip the space $%
LipL(X\times E;F)$ with the norm%
\begin{equation*}
LipL\left( T\right) =\sup_{x\neq y;\left\Vert e\right\Vert =1}\frac{%
\left\Vert T\left( x,e\right) -T\left( y,e\right) \right\Vert }{d\left(
x,y\right) },
\end{equation*}%
and denote by $LipL_{0}(X\times E;F)$ the Banach space of all Lip-Linear
operators from $X\times E$ to $F$ such that $T(0,e)=0$ for each $e\in E$.
The correspondence $T\longleftrightarrow \widehat{T}$ is an isometric
isomorphism between the Banach spaces $LipL_{0}(X\times E;F)$ and $\mathcal{L%
}(X\widehat{\boxtimes }_{\pi }E;F)$ and consequently with the space of
bilinear operators $\mathcal{B}(\mathcal{F}(X)\times E;F).$ If $F=\mathbb{K}$%
, we simply denote $LipL_{0}(X\times E)$. In this case, we obtain a
Lipschitz version of a well-known classical result 
\begin{equation*}
\left( X\widehat{\boxtimes }_{\pi }E\right) ^{\ast }=LipL_{0}(X\times E).
\end{equation*}%
In the linear case, if $v:E\rightarrow F$ is a linear operator, its
corresponding bilinear form $B_{v}:E\times F^{\ast }\rightarrow \mathbb{K}$
is defined as follows 
\begin{equation*}
\begin{array}{ccc}
\left( e,z^{\ast }\right) & \longmapsto & z^{\ast }\left( v\left( e\right)
\right) .%
\end{array}%
\end{equation*}%
This approach can be extended using Lip-linear forms. Specifically, for a
Lipschitz operator $R:X\rightarrow E$, we associate a Lip-linear form $T_{R}$
defined by%
\begin{equation}
\begin{array}{cccc}
T_{R} & X\times E^{\ast } & \longrightarrow & \mathbb{K}%
\end{array}%
\begin{array}{ccc}
\left( x,e^{\ast }\right) & \longmapsto & e^{\ast }\left( R\left( x\right)
\right)%
\end{array}
\label{1,1}
\end{equation}%
In the sequel, we examine the relationship between a Lipschitz operator and
its associated Lip-linear operator in the context of Lipschitz $p$-summing
and integral operators. The theory of summing operators, developed
extensively in various directions, can be naturally extended to Lip-Linear
operators. A key objective in introducing Lip-linear operators is to
generalize this theory and establish identification results. We propose new
definitions of summability and explore the connection between Lip-Linear
operators and their associated operators within these frameworks.

This paper is structured as follows: Section 1 introduces the necessary
definitions and fundamental concepts, including Lipschitz mappings, linear
and bilinear operators, and key identification results. We also recall de
definition of itegral linear and bilinear operators, as well as
two-Lipschitz operators. In Section 2, we define the space $LipL_{0}(X\times
E;F)$ of Lip-linear operators from the cartesian product $X\times E$ to a
Banach space $F$, exploring its properties and establishing relevant
identifications. We show that $LipL_{0}(X\times E;F)$ is isometrically
isomorphic to the space of bilinear operators $\mathcal{B}(\mathcal{F}%
(X)\times E;F)$. Furthermore, when $X$ and $Y$ are pointed metric spaces, we
establish an identification between the space of two-Lipschitz operators $%
Blip_{0}(X\times Y;F)$ and $LipL_{0}(X\times \mathcal{F}(Y);F)$. Section 3
extends the concept of integral operators to Lip-linear operators, allowing
us to establish identification results between the topological dual of $X%
\widehat{\boxtimes }_{\varepsilon }F$ and the space of integral Lip-linear
forms $LipL_{0}^{int}(X\times E)$. We also explore the relationship between
Lipschitz operators and their corresponding Lip-linear operators in the
context of integrability. Additionally, several factorization results are
presented. Finally, Section 4 introduces the concept of dominated $(p,q)$%
-summing Lip-linear operators, establishing key results and exploring their
fundamental properties.

The letters $E,F$ and $G$ stand for Banach spaces over $\mathbb{K}=\mathbb{R}
$ or $\mathbb{C}$. The closed unit ball of $E$ is denoted by $B_{E}$ and its
topological dual by $E^{\ast }.$ The Banach space of all continuous linear
operators from $E$ to $F$ is denoted by $\mathcal{L}(E;F)$, and $\mathcal{B}%
(E\times F;G)$ denotes the Banach space of all continuous bilinear
operators, equipped with the standard sup norm. Let $E\widehat{\otimes }%
_{\pi }F$ and $E\widehat{\otimes }_{\varepsilon }F$ be the Banach spaces of
projective and injective tensor products of $E$ and $F,$ respectively.
Consider the canonical continuous bilinear mapping $\sigma _{2}:E\times
F\longrightarrow E\widehat{\otimes }_{\pi }F,$ defined by%
\begin{equation}
\sigma _{2}\left( x,y\right) =x\otimes y.  \label{1,2}
\end{equation}%
For every continuous bilinear mapping $T:E\times F\longrightarrow G$ there
exists a unique continuous linear operator $\overline{T}:E\widehat{\otimes }%
_{\pi }F\longrightarrow G$ that satisfies $\overline{T}(e\otimes z)=T(e,z)$
for every $(e,z)\in E\times F$ and $\left\Vert \overline{T}\right\Vert
=\left\Vert T\right\Vert $. Taking $G=\mathbb{K}$ , we obtain a canonical
identification of the dual space of the projective tensor product 
\begin{equation*}
(E\widehat{\otimes }_{\pi }F)^{\ast }=\mathcal{B}(E\times F).
\end{equation*}%
Let $X$ and $Y$ be pointed metric spaces with base points denoted by $0$.
The Banach space of Lipschitz functions $T:X\rightarrow E$ that satisfy $%
T\left( 0\right) =0$ is denoted $Lip_{0}(X;E)$ and equipped with the norm%
\begin{equation*}
Lip(T)=\sup_{x\neq y}\frac{\left\Vert T\left( x\right) -T\left( y\right)
\right\Vert }{d\left( x,y\right) }.
\end{equation*}%
The closed unit ball $B_{X^{\#}}$ of $X^{\#}:=Lip_{0}(X;\mathbb{K})$ forms a
compact Hausdorff space under the topology of pointwise convergence on $X$ .
For every $x\in X$, define the evaluation functional $\delta _{x}\in \left(
X^{\#}\right) ^{\ast }$ by $\delta _{x}\left( f\right) =f\left( x\right) .$
The closed linear span of $\left\{ \delta _{x}:x\in X\right\} $ forms a
predual of $X^{\#},$ known as the Lipschitz-free space over $X,$ denoted $%
\mathcal{F}(X),$ as introduced by Godefroy and Kalton in \cite{gk}. In \cite%
{w}, Weaver studied this predual, referring to it as the Arens--Eells space
of $X$ and denoting it \textit{\AE }$(X)$. For $x,y\in X$ and $f\in X^{\#}$,
we have%
\begin{equation*}
\delta _{\left( x,y\right) }\left( f\right) =\delta _{x}\left( f\right)
-\delta _{y}\left( f\right) =f\left( x\right) -f\left( y\right) .
\end{equation*}%
Moreover, $\delta _{\left( x,y\right) }=\delta _{\left( x,0\right) }-\delta
_{\left( y,0\right) }.$ For $m\in \mathcal{F}(X)$, its norm is%
\begin{equation*}
\Vert m\Vert =\inf \left\{ \sum_{i=1}^{n}|\lambda _{i}|d(x_{i},y_{i}):\
m=\sum_{i=1}^{n}\lambda _{i}\delta _{\left( x_{i},y_{i}\right) }\right\}
\end{equation*}%
For any $T\in Lip_{0}(X;E)$, there is a unique linearization $T_{L}\in 
\mathcal{L}(\mathcal{F}(X);E)$ such that%
\begin{equation*}
T=T_{L}\circ \delta _{X},
\end{equation*}%
with $\Vert T_{L}\Vert =Lip(T)$. This yields the isometric identification%
\begin{equation}
Lip_{0}(X;E)=\mathcal{L}(\mathcal{F}(X);E).  \label{13}
\end{equation}%
If $E$ and $F$ are Banach space and $T:E\rightarrow F$ is a linear operator,
then the corresponding linear operator $T_{L}$ is given by%
\begin{equation}
T_{L}=T\circ \beta _{E},  \label{14}
\end{equation}%
where $\beta _{E}:\mathcal{F}\left( E\right) \rightarrow E$ is linear
quotient map verifying $\beta _{E}\circ \delta _{E}=id_{E}$ and $\left\Vert
\beta _{E}\right\Vert \leq 1$ (see \cite[Lemma 2.4]{gk} for more details
about $\beta _{E}$). We denote by $X\boxtimes E$ the Lipschitz tensor
product of $X$ and $E,$ as introduced and studied by Cabrera-Padilla et al.
in \cite{ccjv}. This space is spanned by the functionals $\delta _{\left(
x,y\right) }\boxtimes e$ on $Lip_{0}(X;E^{\ast }),$ defined by $\delta
_{\left( x,y\right) }\boxtimes e(f)=\langle f(x)-f(y),e\rangle $ for $x,y\in
X$ and $e\in E.$ The Lipschitz projective and injective norms on $X\boxtimes
E$ are given by 
\begin{equation*}
\pi (u)=\inf \left\{ \sum_{i=1}^{n}d(x_{i},y_{i})\Vert e_{i}\Vert \right\} 
\text{, }\varepsilon \left( u\right) =\sup_{f\in B_{X},\left\Vert e^{\ast
}\right\Vert =1}\left\{ \left\vert \sum_{i=1}^{n}\left( f\left( x_{i}\right)
-f\left( y_{i}\right) \right) e^{\ast }\left( e_{i}\right) \right\vert
\right\} ,
\end{equation*}%
where the infimum and supremum are taken over all representations of $u$ of
the form $\ u=\sum_{i=1}^{n}\delta _{\left( x_{i},y_{i}\right) }\boxtimes
e_{i}.$ We denote by $X\widehat{\boxtimes }_{\pi }E$ and $X\widehat{%
\boxtimes }_{\varepsilon }E$ the completions of $X\boxtimes E$ with respect
to $\pi \left( \cdot \right) $ and $\varepsilon \left( \cdot \right) ,$
respectively. According to \cite[Proposition 6.7 and ]{ccjv}, the
isometrically identifications%
\begin{equation}
X\widehat{\boxtimes }_{\pi }E=\mathcal{F}\left( X\right) \widehat{\otimes }%
_{\pi }E\text{ and }X\widehat{\boxtimes }_{\varepsilon }E=\mathcal{F}\left(
X\right) \widehat{\otimes }_{\varepsilon }E  \label{1,5}
\end{equation}%
hold via the mapping 
\begin{equation}
I\left( u\right) =I\left( \dsum\limits_{i=1}^{n}\delta _{\left(
x_{i};y_{i}\right) }\boxtimes e_{i}\right) =\dsum\limits_{i=1}^{n}\delta
_{\left( x_{i};y_{i}\right) }\otimes e_{i}.  \label{1,7}
\end{equation}%
This leads to the following isometric identifications%
\begin{equation}
\mathcal{B}\left( \mathcal{F}(X)\times E;F\right) =\mathcal{L}\left( 
\mathcal{F}(X)\widehat{\otimes }_{\pi }E;F\right) =\mathcal{L}\left( X%
\widehat{\boxtimes }_{\pi }E;F\right) .  \label{1,8}
\end{equation}%
Let $E,F$ and $G$ be Banach spaces. In \cite{v}, the author introduced the
concept of integral multilinear operators. A bilinear operator $B:E\times
F\rightarrow G$ is said to be integral (in the sense of Grothendieck) if
there exists a constant $C>0$ such that for any sequences $\left(
x_{i}\right) _{i=1}^{n}\subset E$, $\left( y_{i}\right) _{i=1}^{n}\subset F$
and $\left( z_{i}^{\ast }\right) _{i=1}^{n}\subset G^{\ast }$ we have 
\begin{equation}
\left\vert \dsum\limits_{i=1}^{n}\left\langle B(x_{i},y_{i}),z_{i}^{\ast
}\right\rangle \right\vert \leq C\sup_{x^{\ast }\in B_{E^{\ast }},y^{\ast
}\in B_{F^{\ast }}}\left\Vert \dsum\limits_{i=1}^{n}x^{\ast }\left(
x_{i}\right) y^{\ast }\left( y_{i}\right) z_{i}^{\ast }\right\Vert _{G^{\ast
}}.  \label{1,9}
\end{equation}%
The class of all integral Lip-Linear operators is denoted by $\mathcal{BI}%
(E\times F;G)$, which forms a Banach space with the integral norm%
\begin{equation*}
\mathfrak{I}_{int}^{2}(B)=\inf \left\{ C:C\text{ satisfies }(\ref{1,9}%
)\right\} .
\end{equation*}%
The same definition applies to linear operators, a concept originally
introduced in \cite{g}. We denote by $\mathcal{I}(E;F)$ the Banach space of
all integral linear operators and by $\mathfrak{I}(\cdot )$ its norm.
Furthermore, we have the identification (see, for example, \cite[Proposition
2.6]{v} and \cite[P. 59]{ryn})%
\begin{equation}
\mathcal{BI}(E\times F;G)=\mathcal{I}(E\widehat{\otimes }_{\varepsilon }F;G)
\label{1,10}
\end{equation}%
as well as the duality relation 
\begin{equation}
\left( E\widehat{\otimes }_{\varepsilon }F\right) ^{\ast }=\mathcal{BI}%
(E\times F).  \label{1,11}
\end{equation}%
We recall the definition of a two-Lipschitz operator $\left( \text{see \cite%
{hadt} and \cite{d}}\right) $. Let $X,Y$ be pointed metric spaces, and let $%
F $ be Banach space. A mapping $T:X\times Y\longrightarrow F$ is called a
two-Lipschitz operator if there exists a constant $C>0$ such that, for all $%
x,x^{\prime }\in X$ and $y,y^{\prime }\in Y$, 
\begin{equation*}
\left\Vert T(x,y)-T(x,y^{\prime })-T(x^{\prime },y)+T(x^{\prime },y^{\prime
})\right\Vert \leq Cd(x,x^{\prime })d(y,y^{\prime }),
\end{equation*}%
We denote by $BLip_{0}(X\times Y;F)$ the Banach space of all two-Lipschitz
operators $T$ satisfying $T(0,y)=T\left( x,0\right) =0$ for ever $x\in X$
and $y\in Y$. The norm on this space is given by%
\begin{equation*}
Blip\left( T\right) =\sup_{x\neq x^{\prime };y\neq y^{\prime }}\frac{%
\left\Vert T(x,y)-T(x,y^{\prime })-T(x^{\prime },y)+T(x^{\prime },y^{\prime
})\right\Vert }{d(x,x^{\prime })d(y,y^{\prime })}.
\end{equation*}

\section{\textsc{The Class of Lip-Linear operators}}

Let $X$ be a pointed metric space and $E$ a Banach space. We define
Lip-linear operators on the cartesian product $X\times E$ as mappings that
combine the properties of Lipschitz and linear operators. These operators
can be interpreted either as linear mappings acting on the Lipschitz tensor
product $X\boxtimes E,$ or as bilinear mappings acting on $\mathcal{F}\left(
X\right) \times E$. Among our results, we show that every two-Lipschitz
operator defined on the cartesian product of two pointed metric spaces $%
X\times Y$ can be factored through a Lip-linear operator defined on $X\times 
\mathcal{F}\left( Y\right) $.

\begin{definition}
Let $X$ be a pointed metric space, and let $E,F$ be Banach spaces. A mapping 
$T:X\times E\longrightarrow F$ is called a Lip-Linear operator if there is a
constant $C>0$ such that for all $x,y\in X$ and $e\in E$, 
\begin{equation}
\left\Vert T(x,e)-T(y,e)\right\Vert \leq Cd(x,y)\left\Vert e\right\Vert ,
\label{21}
\end{equation}%
and $T(x,\cdot )$ is a linear operator for each fixed $x\in X$.
\end{definition}

By definition, $T(x,0)=0$ for every $x\in X$. We denote by $LipL_{0}(X\times
E;F)$ the space of all Lip-Linear operators from $X\times E$ to $F$ such
that $T(0,e)=0$ for every $e\in E$. For $T\in LipL_{0}(X\times E;F),$ define 
\begin{equation*}
LipL\left( T\right) :=\inf \left\{ C:\text{ }C\text{ satisfies~ }(\ref{21}%
)\right\} =\sup_{x\neq y;\left\Vert e\right\Vert =1}\frac{\left\Vert T\left(
x,e\right) -T\left( y,e\right) \right\Vert }{d\left( x,y\right) }.
\end{equation*}%
It can be shown that $(LipL_{0},LipL\left( \cdot \right) )$ is a Banach
space. The proof is omitted. Note that, every Lip-Linear operator is
two-Lipschitz, with $Blip(T)\leq LipL(T).$ More generally, if $E,F$ and $G$
be Banach spaces, the following inclusions hold%
\begin{equation}
\mathcal{B}\left( E\times F;G\right) \subset LipL_{0}\left( E\times
F;G\right) \subset BLip_{0}(E\times F;G).  \label{22}
\end{equation}%
and%
\begin{equation*}
Blip(T)\leq LipL(T)\leq \Vert T\Vert .
\end{equation*}%
Indeed, by definition, the inclusions in (\ref{22}) become evident. Let $%
T\in \mathcal{B}\left( E\times F;G\right) $, we have%
\begin{eqnarray*}
Blip(T) &=&\sup_{x\neq y\text{ }x^{\prime }\neq y^{\prime }}\frac{\left\Vert
T\left( x,y\right) -T\left( x,y^{\prime }\right) -T\left( x^{\prime
},y\right) +T\left( x^{\prime },y^{\prime }\right) \right\Vert }{\left\Vert
x-x^{\prime }\right\Vert \left\Vert y-y^{\prime }\right\Vert } \\
&=&\sup_{x\neq y\text{ }x^{\prime }\neq y^{\prime }}\frac{\left\Vert T\left(
x,y-y^{\prime }\right) -T\left( x^{\prime },y-y^{\prime }\right) \right\Vert 
}{\left\Vert x-x^{\prime }\right\Vert \left\Vert y-y^{\prime }\right\Vert }
\end{eqnarray*}%
for $e=y-y^{\prime }$ with $\left\Vert e\right\Vert =1$, this becomes%
\begin{equation*}
Blip(T)\leq \sup_{x\neq y\text{ }\left\Vert e\right\Vert =1}\frac{\left\Vert
T\left( x,e\right) -T\left( x^{\prime },e\right) \right\Vert }{\left\Vert
x-x^{\prime }\right\Vert }=LipL(T).
\end{equation*}%
again, 
\begin{eqnarray*}
LipL(T) &=&\sup_{x\neq y\text{ }\left\Vert e\right\Vert =1}\frac{\left\Vert
T\left( x-x^{\prime },e\right) \right\Vert }{\left\Vert x-x^{\prime
}\right\Vert } \\
&\leq &\sup_{x\neq y\text{ }\left\Vert e\right\Vert =1}\frac{\left\Vert
T\right\Vert \left\Vert x-x^{\prime }\right\Vert \left\Vert e\right\Vert }{%
\left\Vert x-x^{\prime }\right\Vert }=\left\Vert T\right\Vert .
\end{eqnarray*}

\begin{proposition}
Let $R:X\rightarrow G$ be a Lipschitz operator and $v:E\rightarrow F$ a
linear operator. For any bilinear operator $B:G\times F\rightarrow H$, the
operator $T=B\circ (R,v)$ is Lip-Linear, and%
\begin{equation*}
LipL\left( T\right) \leq \left\Vert B\right\Vert Lip\left( R\right)
\left\Vert v\right\Vert .
\end{equation*}%
This means that $\mathcal{B}\circ (Lip_{0},\mathcal{L})\left( X\times
E;H\right) \subset LipL_{0}\left( X\times E;H\right) $.
\end{proposition}

\begin{proof}
For $x\in X,$ the operator $T(x,\cdot )=B(u\left( x\right) ,\cdot )$ is
linear. Let $x,y\in X$ and $e\in E,$ we have%
\begin{eqnarray*}
\left\Vert T(x,e)-T(y,e)\right\Vert &=&\left\Vert B(R\left( x\right)
,v\left( e\right) )-B(R\left( y\right) ,v\left( e\right) )\right\Vert \\
&=&\left\Vert B(R\left( x\right) -R\left( y\right) ,v\left( e\right)
)\right\Vert \\
&\leq &\left\Vert B\right\Vert \left\Vert R\left( x\right) -R\left( y\right)
\right\Vert \left\Vert v\left( e\right) \right\Vert \\
&\leq &\left\Vert B\right\Vert Lip\left( R\right) \left\Vert v\right\Vert
d\left( x,y\right) \left\Vert e\right\Vert .
\end{eqnarray*}%
Thus, $T=B(R,v)$ is Lip-Linear, and 
\begin{equation*}
LipL\left( T\right) \leq \left\Vert B\right\Vert Lip\left( R\right)
\left\Vert v\right\Vert .
\end{equation*}
\end{proof}

For the mapping $T:X\times E\longrightarrow F,$ we define the associated
operators $A_{T}:x\mapsto A_{T}\left( x\right) $ and $B_{T}:e\mapsto
B_{T}\left( e\right) ,$ where $A_{T}(x)\left( e\right) =T(x,e)$ and $%
B_{T}(e)\left( x\right) =T(x,e)$ for every $x\in X$ and $e\in E.$ This
yields the following characterization of a Lip-linear operator.

\begin{proposition}
For a mapping $T:X\times E\longrightarrow F$, the following statements are
equivalent.

1) $T\in LipL_{0}(X\times E;F)$.

2) The Lipschitz operator $A_{T}\in Lip_{0}(X;\mathcal{L}(E;F))$

3) The linear operator $B_{T}\in \mathcal{L}(E;Lip_{0}(X;F)).$

In this case, 
\begin{equation}
LipL(T)=Lip\left( A_{T}\right) =\left\Vert B_{T}\right\Vert .  \label{2.3}
\end{equation}
\end{proposition}

\begin{proof}
We verify the equality $(\ref{2.3})$ as follows:%
\begin{eqnarray*}
Lip\left( A_{T}\right) &=&\sup_{x\neq y}\frac{\left\Vert A_{T}\left(
x\right) -A_{T}\left( y\right) \right\Vert }{d\left( x,y\right) }%
=\sup_{\left\Vert e\right\Vert =1}\sup_{x\neq y}\frac{\left\Vert A_{T}\left(
x\right) \left( e\right) -A_{T}\left( y\right) \left( e\right) \right\Vert }{%
d\left( x,y\right) } \\
&=&\sup_{\left\Vert e\right\Vert =1}\sup_{x\neq y}\frac{\left\Vert T\left(
x,e\right) -T\left( y,e\right) \right\Vert }{d\left( x,y\right) }=LipL(T).
\end{eqnarray*}%
For the second part, we compute%
\begin{eqnarray*}
\left\Vert B_{T}\right\Vert &=&\sup_{\left\Vert e\right\Vert =1}Lip\left(
B_{T}\left( e\right) \right) =\sup_{\left\Vert e\right\Vert =1}\sup_{x\neq y}%
\frac{\left\Vert B_{T}\left( e\right) \left( x\right) -B_{T}\left( e\right)
\left( y\right) \right\Vert }{d\left( x,y\right) } \\
&=&\sup_{\left\Vert e\right\Vert =1}\sup_{x\neq y}\frac{\left\Vert T\left(
x,e\right) -T\left( y,e\right) \right\Vert }{d\left( x,y\right) }=LipL(T).
\end{eqnarray*}
\end{proof}

As an immediate consequence of the last Proposition, we have the following
result.

\begin{corollary}
Let $X$ be a pointed metric space, and let $E$ and $F$ be Banach spaces. For
a two-Lipschitz mapping $T:X\times E\longrightarrow F$, the following
statements are equivalent.

1) $T\in LipL_{0}(X\times E;F)$.

2) $A_{T}\left( x\right) \in \mathcal{L}(E;F)$ for every fixed $x\in X.$

In this case, we have $LipL(T)=Blip(T).$
\end{corollary}

The following result is straightforward.

\begin{proposition}
Let $X$ and $X_{0}$ be pointed metric spaces, and let $E,F,E_{0},F_{0}$ be
Banach spaces. If $R\in Lip_{0}(X;X_{0})$, $v\in \mathcal{L}(E;E_{0})$, $%
T\in LipL_{0}(X_{0}\times E_{0};F_{0})$ and $w\in \mathcal{L}(F_{0};F),$ we
have 
\begin{equation*}
w\circ T\circ (R,v)\in LipL_{0}(X,E;F).
\end{equation*}%
Moreover, 
\begin{equation*}
LipL(w\circ T\circ (R,v))\leq \left\Vert w\right\Vert LipL(T)Lip(R)\Vert
v\Vert .
\end{equation*}
\end{proposition}

\begin{example}
1) Consider the functions $f\in X^{\#}$, $e^{\ast }\in E^{\ast }$ and $z\in
F.$ Define the mapping ${f\cdot e^{\ast }\cdot z}:X\times E\longrightarrow F$
by%
\begin{equation*}
{f\cdot e^{\ast }\cdot z}(x,e)=f(x)e^{\ast }(e)z.
\end{equation*}%
A simple calculation shows that this mapping is Lip-Linear and 
\begin{equation*}
LipL(f\cdot g\cdot e)=Lip(f)\Vert e^{\ast }\Vert \Vert z\Vert .
\end{equation*}%
2) Consider the map $R\in Lip_{0}(X;F)$ and $e^{\ast }\in E^{\ast }$. Define
the mapping $T_{R,e^{\ast }}:X\times E\longrightarrow F$ by $T_{R,e^{\ast
}}(x,e)=R(x)e^{\ast }(e).$ Then, a simple calculation shows that this
mapping is Lip-Linear and 
\begin{equation*}
LipL(T_{f,e^{\ast }})=Lip(R)\Vert e^{\ast }\Vert .
\end{equation*}%
3) The mapping $\Psi :X\times Lip_{0}\left( X;E\right) \rightarrow E$
defined by%
\begin{equation*}
\Psi \left( x,R\right) =R\left( x\right) ,
\end{equation*}%
is Lip-Linear and has the norm $1$.
\end{example}

\begin{lemma}
\label{Lemme1}Let $X$ be a pointed metric space, and let $E,F$ be Banach
spaces. Let $v:E\rightarrow F$ be a linear operator. The operator $\delta
_{X}\boxtimes v:X\times E\rightarrow X\widehat{\boxtimes }_{\pi }F$ defined
by 
\begin{equation*}
\delta _{X}\boxtimes v\left( x,e\right) =\delta _{\left( x,0\right)
}\boxtimes v\left( e\right) ,
\end{equation*}%
is Lip-linear, and we have%
\begin{equation*}
LipL\left( \delta _{X}\boxtimes v\right) =\left\Vert v\right\Vert .
\end{equation*}
\end{lemma}

\begin{proof}
It is clear that $\delta _{X}\boxtimes v$ is linear for the second variable,
and 
\begin{equation*}
\delta _{X}\boxtimes v\left( 0,e\right) =\delta _{\left( 0,0\right)
}\boxtimes v\left( e\right) =0,
\end{equation*}%
for every $e\in E$. Let $x,y\in X$ and $e\in E$ then%
\begin{eqnarray*}
\left\Vert \delta _{X}\boxtimes v\left( x,e\right) -\delta _{X}\boxtimes
v\left( y,e\right) \right\Vert &=&\left\Vert \delta _{\left( x,0\right)
}\boxtimes v\left( e\right) -\delta _{\left( y,0\right) }\boxtimes v\left(
e\right) \right\Vert \\
&=&\left\Vert \left( \delta _{\left( x,0\right) }-\delta _{\left( y,0\right)
}\right) \boxtimes v\left( e\right) \right\Vert \\
&=&\left\Vert \delta _{\left( x,y\right) }\boxtimes v\left( e\right)
\right\Vert \\
&=&\left\Vert \delta _{\left( x,y\right) }\right\Vert \left\Vert v\left(
e\right) \right\Vert \\
&\leq &\left\Vert v\right\Vert d\left( x,y\right) \left\Vert e\right\Vert ,
\end{eqnarray*}%
thus, $\delta _{X}\boxtimes v$ is Lip-linear\ and 
\begin{eqnarray*}
LipL\left( \delta _{X}\boxtimes v\right) &=&\sup_{x\neq y,\left\Vert
e\right\Vert =1}\frac{\left\Vert \delta _{X}\boxtimes v\left( x,e\right)
-\delta _{X}\boxtimes v\left( y,e\right) \right\Vert }{d\left( x,y\right) }
\\
&=&\sup_{x\neq y,\left\Vert e\right\Vert =1}\frac{\left\Vert \delta _{\left(
x,y\right) }\right\Vert \left\Vert v\left( e\right) \right\Vert }{d\left(
x,y\right) } \\
&=&\sup_{x\neq y,\left\Vert e\right\Vert =1}\frac{d\left( x,y\right)
\left\Vert v\left( e\right) \right\Vert }{d\left( x,y\right) }=\left\Vert
v\right\Vert .
\end{eqnarray*}
\end{proof}

Consider the mapping $\sigma ^{LipL}:X\times E\rightarrow X\widehat{%
\boxtimes }_{\pi }E$ defined by 
\begin{equation*}
\sigma ^{LipL}(x,e)=\delta _{\left( x,0\right) }\boxtimes e.
\end{equation*}%
It is straightforward to verify that $\sigma ^{LipL}=\delta _{X}\boxtimes
id_{E}.$ By Lemma $\ref{Lemme1},$\ we conclude that $\sigma ^{LipL}\in
LipL_{0}(X\times E;X\widehat{\boxtimes }_{\pi }E)$ and $LipL(\sigma
^{LipL})=1.$ Let $T\in LipL_{0}(X\times E;F),$ from $(\ref{22}),$ we know
that $T$ is two-Lipschitz. The factorization theorem for two-Lipschitz
operators, as presented in \cite[Remark 2.7]{hadt}, guarantees the existence
of a unique bounded linear operator $\widetilde{T}:\mathcal{F}(X)\widehat{%
\otimes }_{\pi }\mathcal{F}(E)\rightarrow F$ such that%
\begin{equation*}
T=\widetilde{T}\circ \sigma _{2}\circ \left( \delta _{X},\delta _{E}\right) ,
\end{equation*}%
where $\sigma _{2}$ is the canonical continuous bilinear map defined in $%
\left( \ref{1,2}\right) $. This implies that the following diagram commutes:%
\begin{equation*}
\begin{array}{ccc}
X\times E &  &  \\ 
\left( \delta _{X},\delta _{E}\right) \downarrow & T\searrow &  \\ 
\mathcal{F}(X)\times \mathcal{F}(E) &  & F \\ 
\sigma _{2}\downarrow & \widetilde{T}\nearrow &  \\ 
\mathcal{F}(X)\widehat{\otimes }_{\pi }\mathcal{F}(E) &  & 
\end{array}%
\end{equation*}%
The following theorem provides a further factorization result for Lip-linear
operators.

\begin{theorem}
\label{IdentLinearisation}For every Lip-Linear operator $T\in
LipL_{0}(X\times E;F),$ there exists a unique bounded linear operator $%
\widehat{T}:X\widehat{\boxtimes }_{\pi }E\rightarrow F$ defined by 
\begin{equation*}
\widehat{T}(\delta _{\left( x,y\right) }\boxtimes e)=T(x,e)-T(y,e),
\end{equation*}%
for all $x,y\in X$ and $e\in E.$ In other words, $T=\widehat{T}\circ \sigma
^{LipL},$ ensuring that the following diagram commutes 
\begin{equation*}
\begin{array}{ccc}
X\times E & \overset{T}{\longrightarrow } & F \\ 
\sigma ^{LipL}\downarrow & \nearrow \widehat{T} &  \\ 
X\widehat{\boxtimes }_{\pi }E &  & 
\end{array}%
\end{equation*}%
In this case, we have%
\begin{equation*}
LipL\left( T\right) =\left\Vert \widehat{T}\right\Vert .
\end{equation*}
\end{theorem}

\begin{proof}
Let us first show that $\widehat{T}$ is bounded with respect to the
Lipschitz projective norm on $X\boxtimes E$. For any element $%
u=\sum_{i=1}^{n}\delta _{\left( x_{i},y_{i}\right) }\boxtimes e_{i}\in
X\boxtimes E,$ we have 
\begin{eqnarray*}
\left\Vert \widehat{T}\left( u\right) \right\Vert &=&\left\Vert
\sum_{i=1}^{n}T\left( x_{i},e_{i}\right) -T\left( y_{i},e_{i}\right)
\right\Vert \\
&\leq &\sum_{i=1}^{n}\left\Vert T\left( x_{i},e_{i}\right) -T\left(
y_{i},e_{i}\right) \right\Vert \\
&\leq &LipL\left( T\right) \sum_{i=1}^{n}d\left( x_{i},y_{i}\right)
\left\Vert e_{i}\right\Vert .
\end{eqnarray*}%
Since this holds for every representation of $u,$ it follows that%
\begin{equation*}
\Vert \widehat{T}(u)\Vert \leq LipL(T)\pi (u).
\end{equation*}%
Thus, $\widehat{T}$ is bounded, and we have $\left\Vert \widehat{T}%
\right\Vert \leq LipL(T).$ Conversely, we have 
\begin{eqnarray*}
\Vert T(x,e)-T(y,e)\Vert &=&\Vert \widehat{T}(\delta _{\left( x,y\right)
}\boxtimes e)\Vert \\
&\leq &\Vert \widehat{T}\Vert d(x,y)\Vert e\Vert
\end{eqnarray*}%
which implies that $LipL(T)\leq \Vert \widehat{T}\Vert .$ Therefore, the
operator $\widehat{T}:X\boxtimes _{\pi }E\rightarrow F$ has a unique
extension to $X\widehat{\boxtimes }_{\pi }E$ with the same norm. The mapping 
$T\mapsto \widehat{T}$ is clearly a linear isometry. For scalars $\alpha $
and $\beta ,$ we have 
\begin{eqnarray*}
\widehat{\alpha S+\beta T}\left( u\right) &=&\widehat{\alpha S+\beta T}%
(\sum_{i=1}^{n}\delta _{\left( x_{i},y_{i}\right) }\boxtimes e_{i}) \\
&=&\sum_{i=1}^{n}\left( \alpha S+\beta T\right) \left( x_{i},e_{i}\right)
-\left( \alpha S+\beta T\right) \left( y_{i},e_{i}\right) \\
&=&\sum_{i=1}^{n}\alpha \left( \widehat{S}\left( \delta _{\left(
x_{i},y_{i}\right) }\boxtimes e_{i}\right) \right) +\beta \widehat{T}\left(
\delta _{\left( x_{i},y_{i}\right) }\boxtimes e_{i}\right) \\
&=&\alpha \widehat{S}(\sum_{i=1}^{n}\delta _{\left( x_{i},y_{i}\right)
}\boxtimes e_{i})+\beta \widehat{T}(\sum_{i=1}^{n}\delta _{\left(
x_{i},y_{i}\right) }\boxtimes e_{i}).
\end{eqnarray*}%
Thus, we conclude that%
\begin{equation*}
\widehat{\alpha S+\beta T}\left( u\right) =\left( \alpha \widehat{S}+\beta 
\widehat{T}\right) \left( u\right) .
\end{equation*}%
It remains to show that this mapping is surjective. Let $A\in \mathcal{L}(X%
\widehat{\boxtimes }_{\pi }E;F)$. The Lip-Linear operator defined by 
\begin{equation*}
T(x,e)=A(\delta _{\left( x,0\right) }\boxtimes e)
\end{equation*}%
satisfies $\widehat{T}=A,$ completing the proof.
\end{proof}

The previous theorem establishes the following isometric identification 
\begin{equation}
LipL_{0}(X\times E;F)=\mathcal{L}(X\widehat{\boxtimes }_{\pi }E;F),
\label{2.4}
\end{equation}%
via the correspondence $T\leftrightarrow \widehat{T}$. Consequently, from $(%
\ref{1,8})$, it follows that $LipL_{0}(X\times E;F)$ is isometrically
isomorphic to $\mathcal{B}(\mathcal{F}(X)\times E;F).$ This implies that for
every Lip-Linear operator $T:X\times E\longrightarrow F$, there exists a
unique bilinear operator $T_{bil}:\mathcal{F}\left( X\right) \times
E\longrightarrow F$ such that the following diagram commutes%
\begin{equation}
\begin{array}{ccccc}
X\times E & \overset{T}{\longrightarrow } & F &  &  \\ 
\delta _{X}\times id_{E}\downarrow & T_{bil}\nearrow & \overline{T_{bil}}%
\uparrow & \nwarrow \widehat{T} &  \\ 
\text{$\mathcal{F}$}\left( X\right) \times E & \overset{\sigma _{2}}{%
\longrightarrow } & \text{$\mathcal{F}$}\left( X\right) \widehat{\otimes }%
_{\pi }E & \overset{I^{-1}}{\longrightarrow } & X\widehat{\boxtimes }_{\pi }E%
\end{array}
\label{2.5}
\end{equation}%
In other words, $T_{bil}=\widehat{T}\circ I^{-1}\circ \sigma _{2}.$ This
implies that the linearization of $T_{LipL}$ satisfies $\overline{T_{bil}}=%
\widehat{T}\circ I^{-1}.$ Taking $F=\mathbb{K}$ in $(\ref{2.4}),$ we find $%
LipL_{0}(X\times E)=(X\widehat{\boxtimes }_{\pi }E)^{\ast }.$ This duality
leads to a new formula for the Lipschitz projective norm. For each $u\in X%
\widehat{\boxtimes }_{\pi }E,$ we have%
\begin{equation*}
\pi \left( u\right) =\sup \left\{ \left\vert \left\langle u,T\right\rangle
\right\vert :T\in LipL_{0}(X\times E)\text{ and }LipL(T)\leq 1\right\} .
\end{equation*}

According to \cite[Page 11]{dk}, and the formula $\left( \ref{13}\right) ,$
we have the following identification.

\begin{corollary}
Let $X$ be a pointed metric space, and let $E,F$ be Banach spaces. The
following isometric identifications hold%
\begin{equation*}
LipL_{0}(X\times E;F)=\mathcal{L}(\mathcal{F}(X);\mathcal{L}\left(
E;F\right) )=Lip_{0}(X;\mathcal{L}\left( E;F\right) ).
\end{equation*}
\end{corollary}

\begin{proposition}
Let $X,Y$ be pointed metric spaces, and $F$ a Banach space. If $T:X\times
Y\rightarrow F$ is a two-Lipschitz operator, then there exists a unique
Lip-Linear operator $T_{LipL}:X\times \mathcal{F}\left( Y\right) \rightarrow
F$ such that%
\begin{equation*}
T=T_{LipL}\circ \left( id_{X},\delta _{Y}\right) .
\end{equation*}%
Conversely, if $T:X\times \mathcal{F}\left( Y\right) \rightarrow F$ is a
Lip-Linear operator, then there exists a unique two-Lipschitz operator $%
T_{Blip}:X\times Y\rightarrow F$ such that%
\begin{equation*}
\left( T_{Blip}\right) _{LipL}=T.
\end{equation*}%
Additionally, we have the following isometric identification%
\begin{equation*}
BLip_{0}\left( X\times Y;F\right) =LipL_{0}(X\times \mathcal{F}(X);F)
\end{equation*}
\end{proposition}

\begin{proof}
Let $T\in BLip_{0}\left( X\times Y;F\right) ,$ and let $\widetilde{T}:%
\mathcal{F}(X)\widehat{\otimes }_{\pi }\mathcal{F}(Y)\rightarrow F$ be its
linearization, i.e., 
\begin{equation*}
T=\widetilde{T}\circ \sigma _{2}\circ \left( \delta _{X},\delta _{Y}\right) .
\end{equation*}%
By $\left( \ref{2.4}\right) ,$ there exists a unique Lip-Linear operator $%
T_{LipL}:X\times \mathcal{F}(Y)\rightarrow F$ such that%
\begin{equation*}
T_{LipL}=\widetilde{T}\circ I\circ \sigma ^{LipL},
\end{equation*}%
ensuring the commutativity of the following diagram%
\begin{equation*}
\begin{array}{ccc}
\mathcal{F}(X)\widehat{\otimes }_{\pi }\text{$\mathcal{F}$}\left( Y\right) & 
\overset{\widetilde{T}}{\rightarrow } & F \\ 
\uparrow I &  & \uparrow T_{LipL} \\ 
X\widehat{\boxtimes }_{\pi }\text{$\mathcal{F}$}\left( Y\right) & \overset{%
\sigma ^{LipL}}{\longleftarrow } & X\times \text{$\mathcal{F}$}\left(
Y\right)%
\end{array}%
\end{equation*}%
where $I$ is the isometric isomorphism defined in $(\ref{1,7}).$ It follows
that $LipL\left( T_{LipL}\right) =\left\Vert \widetilde{T}\right\Vert ,$ and 
\begin{equation*}
T_{LipL}\circ \left( id_{X},\delta _{Y}\right) =\widetilde{T}\circ I\circ
\sigma ^{LipL}\circ \left( id_{X},\delta _{Y}\right) .
\end{equation*}%
We have $\widetilde{T}\circ I\circ \sigma ^{LipL}\circ \left( id_{X},\delta
_{Y}\right) =T.$ Indeed, for $x\in X$ and $e\in E,$ we have%
\begin{eqnarray*}
\widetilde{T}\circ I\circ \sigma ^{LipL}\circ \left( id_{X},\delta
_{Y}\right) \left( x,y\right) &=&\widetilde{T}\circ I\circ \sigma
^{LipL}\circ \left( x,\delta _{y}\right) \\
&=&\widetilde{T}\circ I\circ \left( \delta _{x}\boxtimes \delta _{y}\right)
\\
&=&\widetilde{T}\left( \delta _{x}\otimes \delta _{y}\right) \\
&=&\widetilde{T}\circ \sigma _{2}\circ \left( \delta _{X},\delta _{Y}\right)
\left( x,y\right) \\
&=&T\left( x,y\right) .
\end{eqnarray*}%
For the converse, let $T:X\times \mathcal{F}\left( Y\right) \rightarrow F$
be a Lip-Linear operator. Define%
\begin{equation*}
T_{Blip}=\widehat{T}\circ I\circ \sigma ^{LipL}
\end{equation*}%
where $\widehat{T}:X\widehat{\boxtimes }_{\pi }\mathcal{F}\left( Y\right)
\rightarrow F$ is the linearization of the Lip-Linear operator $T.$ By \cite[%
Remark 2.7]{hadt}, there exists a unique two-Lipschitz operator $%
T_{Blip}:X\times Y\rightarrow F$ such that%
\begin{equation*}
\widehat{T}\circ I^{-1}\circ \sigma _{2}\circ \left( \delta _{X},\delta
_{Y}\right) =T_{Blip}
\end{equation*}%
with $Blip\left( T_{Blip}\right) =\left\Vert \widehat{T}\right\Vert .$ We
have $\widetilde{T_{Blip}}=\widehat{T}\circ I^{-1}$ and 
\begin{equation*}
\left( T_{Blip}\right) _{LipL}=\widetilde{T_{Blip}}\circ I\circ \sigma
^{LipL}=\widehat{T}\circ I^{-1}\circ I\circ \sigma ^{LipL}=T.
\end{equation*}
\end{proof}

\begin{corollary}
Let $T:X\times E\rightarrow F$ be a Lip-Linear operator, $\widehat{T}$ its
linearization and $\widetilde{T}$ its linearization as a two-Lipschitz
operator. We have%
\begin{equation*}
\widetilde{T}=\widehat{T}\circ I^{-1}\circ id_{\text{$\mathcal{F}$}\left(
X\right) }\otimes \beta _{E},
\end{equation*}%
which implies that the following diagram commutes%
\begin{equation*}
\begin{array}{ccc}
\text{$\mathcal{F}$}\left( X\right) \widehat{\otimes }_{\pi }\text{$\mathcal{%
F}$}\left( E\right) & \overset{\widetilde{T}}{\longrightarrow } & F \\ 
id_{\text{$\mathcal{F}$}\left( X\right) }\otimes \beta _{E}\downarrow &  & 
\uparrow \widehat{T} \\ 
\text{$\mathcal{F}$}\left( X\right) \widehat{\otimes }_{\pi }E & \overset{%
I^{-1}}{\longrightarrow } & X\widehat{\boxtimes }_{\pi }E%
\end{array}%
\end{equation*}%
where $\beta _{E}$ is the quotient map given in $\left( \ref{14}\right) .$
\end{corollary}

\begin{proof}
Let $x\in X$ and $e\in E.$ We have%
\begin{eqnarray*}
&&\widehat{T}\circ I^{-1}\circ id_{\text{$\mathcal{F}$}\left( X\right)
}\otimes \beta _{E}\circ \sigma _{2}\circ \left( \delta _{X},\delta
_{E}\right) \left( x,e\right) \\
&=&\widehat{T}\circ I^{-1}\circ id_{\text{$\mathcal{F}$}\left( X\right)
}\otimes \beta _{E}\left( \delta _{\left( x,0\right) }\otimes \delta
_{\left( e,0\right) }\right) \\
&=&\widehat{T}\circ I^{-1}\circ id_{\text{$\mathcal{F}$}\left( X\right)
}\left( \delta _{\left( x,0\right) }\right) \otimes \beta _{E}\left( \delta
_{\left( e,0\right) }\right) \\
&=&\widehat{T}\circ I^{-1}\circ \left( \delta _{\left( x,0\right) }\otimes
e\right) \\
&=&\widehat{T}\circ \left( \delta _{\left( x,0\right) }\boxtimes e\right) =%
\widehat{T}\circ \sigma ^{LipL}\left( x,e\right) =T\left( x,e\right) .
\end{eqnarray*}%
Due to the uniqueness of the bounded linear operator satisfying this
relationship, we conclude that%
\begin{equation*}
\widetilde{T}=\widehat{T}\circ I^{-1}\circ id_{\text{$\mathcal{F}$}\left(
X\right) }\otimes \beta _{E}.
\end{equation*}
\end{proof}

\section{\textsc{Integral Lip-Linear operators and the dual of }$X\widehat{%
\boxtimes }_{\protect\varepsilon }E$}

Let $E$ and $F$ be Banach spaces, and let $E\widehat{\otimes }_{\varepsilon
}F$ denote their injective tensor product. In the classical case, it is
well-known that the dual space of $E\widehat{\otimes }_{\varepsilon }F$
coincides with the space of integral bilinear forms on $E\times F.$
Motivated by this fact, we introduce a new definition of integral Lip-Linear
operators and establish that their linearization operators are integral.
Among our results, we prove that the dual space of $X\widehat{\boxtimes }%
_{\varepsilon }E$ coincides with the space of integral Lip-Linear forms. Let 
$X$ be a pointed metric space, and let $E,F$ be Banach space. A Lip-Linear
operator $T:X\times E\rightarrow F$ is said to be integral if there exists a
constant $C>0$ such that for any $x_{1},\ldots ,x_{n},y_{1},\ldots ,y_{n}$
in $X$, $e_{1},\ldots ,e_{n}$ in $E$ and $z_{1}^{\ast },...,z_{n}^{\ast }\in
F^{\ast }$ we have 
\begin{equation}
\left\vert \dsum\limits_{i=1}^{n}\left\langle
T(x_{i},e_{i})-T(y_{i},e_{i}),z_{i}^{\ast }\right\rangle \right\vert \leq
C\sup_{f\in B_{X^{\#}},\left\Vert e^{\ast }\right\Vert =1}\left\Vert
\dsum\limits_{i=1}^{n}\left( f(x_{i})-f(y_{i})\right) e^{\ast }\left(
e_{i}\right) z_{i}^{\ast }\right\Vert _{F^{\ast }}.  \label{31}
\end{equation}%
The class of all integral Lip-Linear operators is denoted by $%
LipL_{0}^{int}(X\times E;F)$. For such operators, we define%
\begin{equation*}
\mathfrak{I}_{int}^{LipL}(T)=\inf \left\{ C:C\text{ satisfies }(\ref{31}%
)\right\} .
\end{equation*}%
It is straightforward to verify that the pair $\left( LipL_{0}^{int}(X\times
E;F),\mathfrak{I}_{int}^{LipL}(\cdot )\right) $ forms a Banach space. The
importance of this new definition lies in its connection with the
linearization operators, allowing the extension of certain classical
results. The next theorem establishes this relationship.

\begin{theorem}
\label{TheoremFond} Let $X$ be a pointed metric space, and let $E,F$ be
Banach spaces. The following properties are equivalent:

1) The Lip-Linear operator $T$ is integral.

2) The linearization $\widehat{T}:X\widehat{\boxtimes }_{\varepsilon
}E\longrightarrow F$ is intgral.
\end{theorem}

\begin{proof}
$1)\Longrightarrow 2):$ Let $T:X\times E\longrightarrow F$ be an itegral
Lip-Linear operator. For $u_{j}=\sum_{i=1}^{n}\delta _{\left(
x_{i}^{j},y_{i}^{j}\right) }\boxtimes e_{i}^{j}\in X\boxtimes _{\varepsilon
}E$ and $z_{j}^{\ast }\in F^{\ast }\left( 1\leq j\leq m\right) .$ We have 
\begin{eqnarray*}
\left\vert \sum_{j=1}^{m}\left\langle \widehat{T}(u_{j}),z_{j}^{\ast
}\right\rangle \right\vert &=&\left\vert
\sum_{j=1}^{m}\sum_{i=1}^{n}\left\langle
T(x_{i}^{j},e_{i}^{j})-T(y_{i}^{j},e_{i}^{j}),z_{j}^{\ast }\right\rangle
\right\vert \\
&\leq &\mathfrak{I}_{int}^{LipL}\left( T\right) \sup_{f\in
B_{X^{\#}},\left\Vert e^{\ast }\right\Vert =1}\left\Vert
\sum_{j=1}^{m}\sum_{i=1}^{n}\left( f(x_{i}^{j})-f(y_{i}^{j})\right) e^{\ast
}\left( e_{i}^{j}\right) z_{j}^{\ast }\right\Vert _{F^{\ast }} \\
&\leq &\mathfrak{I}_{int}^{LipL}\left( T\right) \sup_{f\in
B_{X^{\#}},\left\Vert e^{\ast }\right\Vert =1}\left\Vert
\sum_{j=1}^{m}f\boxtimes e^{\ast }\left( u_{j}\right) z_{j}^{\ast
}\right\Vert _{F^{\ast }}.
\end{eqnarray*}%
Since 
\begin{equation*}
\sup_{f\in B_{X^{\#}},\left\Vert e^{\ast }\right\Vert =1}\left\Vert
\sum_{j=1}^{m}f\boxtimes e^{\ast }\left( u_{j}\right) z_{j}^{\ast
}\right\Vert _{F^{\ast }}=\sup \left\{ \left\Vert \sum_{j=1}^{m}\phi \left(
u_{j}\right) z_{j}^{\ast }\right\Vert _{F^{\ast }}:\phi \in B_{\left( X%
\widehat{\boxtimes }_{\varepsilon }E\right) ^{\ast }}\right\}
\end{equation*}%
It follows that $\widehat{T}$ is integral, with $\mathfrak{I}\left( \widehat{%
T}\right) \leq \mathfrak{I}_{int}^{LipL}(T).$

$2)\Longrightarrow 1):$ Suppose $\widehat{T}:X\widehat{\boxtimes }%
_{\varepsilon }E\longrightarrow F$ is integral. Let $\left( x_{i}\right)
_{i=1}^{n},\left( y_{i}\right) _{i=1}^{n}\subset X$, $\left( e_{i}\right)
_{i=1}^{n}\subset E$ and $\left( z_{i}^{\ast }\right) _{i=1}^{n}\subset
F^{\ast }$. Then%
\begin{eqnarray*}
\left\vert \dsum\limits_{i=1}^{n}\left\langle
T(x_{i},e_{i})-T(y_{i},e_{i}),z_{i}^{\ast }\right\rangle \right\vert
&=&\left\vert \sum_{i=1}^{n}\left\langle \widehat{T}\left( \delta _{\left(
x_{i},y_{i}\right) }\boxtimes e_{i}\right) ,z_{i}^{\ast }\right\rangle
\right\vert \\
&\leq &\mathfrak{I}\left( \widehat{T}\right) \sup \left\{ \left\Vert
\sum_{i=1}^{n}\phi \left( \delta _{\left( x_{i},y_{i}\right) }\boxtimes
e_{i}\right) z_{i}^{\ast }\right\Vert _{F^{\ast }}:\phi \in B_{\left( X%
\widehat{\boxtimes }_{\varepsilon }E\right) ^{\ast }}\right\} \\
&\leq &\mathfrak{I}\left( \widehat{T}\right) \sup_{f\in
B_{X^{\#}},\left\Vert e^{\ast }\right\Vert =1}\left\Vert
\sum_{i=1}^{n}\left( f(x_{i})-f(y_{i})\right) e^{\ast }\left( e_{i}\right)
z_{i}^{\ast }\right\Vert _{F^{\ast }}.
\end{eqnarray*}%
Thus, $T$ is integral and $\mathfrak{I}_{int}^{LipL}(T)\leq \mathfrak{I}%
\left( \widehat{T}\right) .$
\end{proof}

\begin{corollary}
Let $X$ be a pointed metric space, and let $E,F$ be Banach spaces. We have
the following isometric identification%
\begin{equation}
LipL_{0}^{int}(X\times E;F)=\mathcal{I}\left( X\widehat{\boxtimes }%
_{\varepsilon }E;F\right) .  \label{32}
\end{equation}
\end{corollary}

\begin{proof}
Define the map%
\begin{equation*}
\Gamma :LipL_{0}^{int}(X\times E;F)\rightarrow \mathcal{I}\left( X\widehat{%
\boxtimes }_{\varepsilon }E;F\right)
\end{equation*}%
by $\Gamma \left( T\right) =\widehat{T}.$ By Theorem $\ref{TheoremFond}$, $%
\Gamma $ is well-defined and isometric. We will now show that $\Gamma $ is
surjective. Let $\varphi \in \mathcal{I}\left( X\widehat{\boxtimes }%
_{\varepsilon }E;F\right) .$ Define%
\begin{equation*}
T_{\varphi }\left( x,e\right) =\varphi \left( \delta _{x}\boxtimes e\right) .
\end{equation*}%
Then, $T_{\varphi }$ is Lip-Linear operator. For any $\left( x_{i}\right)
_{i=1}^{n},\left( y_{i}\right) _{i=1}^{n}\subset X,$ $\left( e_{i}\right)
_{i=1}^{n}\subset E$ and $\left( z_{i}^{\ast }\right) _{i=1}^{n}\subset
F^{\ast },$ then%
\begin{eqnarray*}
\left\vert \sum_{i=1}^{n}\left\langle T_{\varphi }\left( x_{i},e_{i}\right)
-T_{\varphi }\left( y_{i},e_{i}\right) ,z_{i}^{\ast }\right\rangle
\right\vert &=&\left\vert \sum_{i=1}^{n}\left\langle \varphi \left( \delta
_{x_{i}}\boxtimes e_{i}\right) -\varphi \left( \delta _{y_{i}}\boxtimes
e_{i}\right) ,z_{i}^{\ast }\right\rangle \right\vert \\
&=&\left\vert \sum_{i=1}^{n}\left\langle \varphi \left( \delta _{\left(
x_{i},y_{i}\right) }\boxtimes e_{i}\right) ,z_{i}^{\ast }\right\rangle
\right\vert
\end{eqnarray*}%
Since $\varphi $ is integral,%
\begin{equation*}
\left\vert \sum_{i=1}^{n}\left\langle T_{\varphi }\left( x_{i},e_{i}\right)
-T_{\varphi }\left( y_{i},e_{i}\right) ,z_{i}^{\ast }\right\rangle
\right\vert \leq \mathfrak{I}\left( \varphi \right) \sup_{f\in
B_{X^{\#}},\left\Vert e^{\ast }\right\Vert =1}\left\Vert
\sum_{i=1}^{n}\left( f(x_{i})-f(y_{i})\right) e^{\ast }\left( e_{i}\right)
z_{i}^{\ast }\right\Vert _{F^{\ast }}
\end{equation*}%
Thus, $T_{\varphi }$ is integral. Finally, for $u=\sum_{i=1}^{n}\delta
_{\left( x_{i},y_{i}\right) }\boxtimes e_{i}$, we have%
\begin{equation*}
\widehat{T_{\varphi }}\left( u\right) =\widehat{T_{\varphi }}\left(
\sum_{i=1}^{n}\delta _{\left( x_{i},y_{i}\right) }\boxtimes e_{i}\right)
=\sum_{i=1}^{n}T_{\varphi }\left( x_{i},e_{i}\right) -T_{\varphi }\left(
y_{i},e_{i}\right)
\end{equation*}%
Substituting $T_{\varphi }\left( x_{i},e_{i}\right) =\varphi \left( \delta
_{x_{i}}\boxtimes e_{i}\right) ,$ we get%
\begin{equation*}
\widehat{T_{\varphi }}\left( u\right) =\sum_{i=1}^{n}\varphi \left( \delta
_{x_{i}}\boxtimes e_{i}\right) -\varphi \left( \delta _{y_{i}}\boxtimes
e_{i}\right) =\sum_{i=1}^{n}\varphi \left( \delta _{\left(
x_{i},y_{i}\right) }\boxtimes e_{i}\right) =\varphi \left( u\right) .
\end{equation*}%
Then, $\widehat{T_{\varphi }}=\varphi $ on $X\boxtimes _{\varepsilon }E.$ By
density, this equality extends to $X\widehat{\boxtimes }_{\varepsilon }E.$
Therefore, we conclude that $\Gamma $ is surjective.
\end{proof}

The following result establishes the relationship between a Lip-linear
operator and its corresponding bilinear operator in the context of
integrability.

\begin{proposition}
\label{ProTandTbil}Let $X$ be a pointed metric space and $E$ a Banach space.
For a Lip-Linear operator $T:X\times E\rightarrow F$, the following
assertions are equivalent

1) The Lip-Linear operator $T$ belongs to $LipL_{0}^{int}(X\times E;F)$.

2) The bilinear operator $T_{bil}$ belongs to $\mathcal{BI}(\mathcal{F}%
\left( X\right) \times E;F)$.

Moreover, we have the equality%
\begin{equation*}
\mathfrak{I}_{int}^{LipL}(T)=\mathfrak{I}_{int}^{2}(T_{bil}).
\end{equation*}

This leads to the following isometric identification%
\begin{equation}
LipL_{0}^{int}(X\times E;F)=\mathcal{BI}(\mathcal{F}\left( X\right) \times
E;F)  \label{33}
\end{equation}
\end{proposition}

\begin{proof}
From the following commutative diagram%
\begin{equation*}
\begin{array}{ccccc}
X\times E & \overset{T}{\longrightarrow } & F &  &  \\ 
\delta _{X}\times id_{E}\downarrow & T_{bil}\nearrow & \overline{T_{bil}}%
\uparrow & \nwarrow \widehat{T} &  \\ 
\text{$\mathcal{F}$}\left( X\right) \times E & \overset{\sigma _{2}}{%
\longrightarrow } & \text{$\mathcal{F}$}\left( X\right) \widehat{\otimes }%
_{\varepsilon }E & \overset{I^{-1}}{\longrightarrow } & X\widehat{\boxtimes }%
_{\varepsilon }E%
\end{array}%
\end{equation*}%
we obtain that $\overline{T_{bil}}=\widehat{T}\circ I^{-1}$ is inegral if
and only if $\widehat{T}$ is integral, from which the result follows
immediately. For the identification $(\ref{33})$, we have by $\left( \ref%
{1,10}\right) $ 
\begin{equation*}
\mathcal{BI}(\mathcal{F}\left( X\right) \times E;F)=\mathcal{I}(\mathcal{F}%
\left( X\right) \widehat{\otimes }_{\varepsilon }E;F)
\end{equation*}%
Using the identifications $\left( \ref{1,5}\right) $ and $(\ref{32})$, we
deduce that%
\begin{equation*}
LipL_{0}^{int}(X\times E;F)=\mathcal{BI}(\mathcal{F}\left( X\right) \times
E;F).
\end{equation*}
\end{proof}

By setting $F=\mathbb{K}$ in $(\ref{33})$, the identifications $(\ref{1,5})$
and $(\ref{1,11})$ yield%
\begin{equation}
\left( X\widehat{\boxtimes }_{\varepsilon }E\right) ^{\ast
}=LipL_{0}^{int}(X\times E).  \label{34}
\end{equation}

Let $R:X\rightarrow E$ be a Lipschitz operator, and let $T_{R}$ be its
associated Lip-Linear form, as defined in $\left( \ref{1,1}\right) $. It is
straightforward to verify that for $m\in \mathcal{F}$$\left( X\right) $ and $%
e^{\ast }\in E^{\ast }$, then%
\begin{equation*}
\widehat{T_{R}}\left( m\boxtimes e^{\ast }\right) =e^{\ast }\left(
R_{L}\left( m\right) \right) ,
\end{equation*}%
where $\widehat{T_{R}}$ and $R_{L}$ are the linearizations of $T_{R}$ and $%
R, $ respectively. The definition of Lipschitz integral operators is given
in \cite{cab}$.$ According to \cite[Proposition 2.4]{cab}, the Lipschitz
operator $R$ is Lipschitz integral if and only if its linearization $R_{L}$
is also integral. Furthermore, by \cite[Definition 2.2.]{v}, $R_{L}$ is
integral, if and only if the bilinear operator%
\begin{equation*}
\begin{array}{cccc}
T_{R_{L}} & \text{$\mathcal{F}$}\left( X\right) \times E^{\ast } & 
\longrightarrow & \mathbb{K} \\ 
& \left( m,e^{\ast }\right) & \longmapsto & e^{\ast }\left( R_{L}\left(
m\right) \right)%
\end{array}%
\end{equation*}%
is integral. Based on this, we can state the following result.

\begin{proposition}
For a Lipschitz operator $u:X\longrightarrow E$, the following statements
are equivalent.

1) $R\in Lip_{0GI}\left( X;E\right) $.

2) The Lip-Linear form $T_{R}:X\times E^{\ast }\rightarrow \mathbb{K}$ is
integral.

In this case,%
\begin{equation*}
Lip_{GI}\left( R\right) =\mathfrak{I}_{int}^{LipL}(T_{R}).
\end{equation*}
\end{proposition}

\begin{proof}
Let $R_{L}:\mathcal{F}\left( X\right) \rightarrow E$ be the linearization of 
$R.$ Its corresponding bilinear form is defined as%
\begin{equation*}
\begin{array}{cccc}
B_{R_{L}}: & \text{$\mathcal{F}$}\left( X\right) \times E^{\ast } & 
\rightarrow & \mathbb{K} \\ 
& \left( m,e^{\ast }\right) & \mapsto & e^{\ast }\left( R_{L}\left( m\right)
\right)%
\end{array}%
\end{equation*}%
On the other hand, by $\left( \ref{2.5}\right) $, the bilinear form
associated with $T_{R}$ is given by%
\begin{equation*}
\left( T_{R}\right) _{bil}:\text{$\mathcal{F}$}\left( X\right) \times
E^{\ast }\rightarrow \mathbb{K}
\end{equation*}%
where 
\begin{eqnarray*}
\left( T_{R}\right) _{bil}\left( m,e^{\ast }\right) &=&\widehat{T_{R}}\circ
I^{-1}\circ \sigma \left( m,e^{\ast }\right) \\
&=&\widehat{T_{R}}\left( m\boxtimes e^{\ast }\right) \\
&=&e^{\ast }\left( R_{L}\left( m\right) \right)
\end{eqnarray*}%
Therefore, we conclude that $B_{R_{L}}=\left( T_{R}\right) _{bil},$ which
establishes the equivalence.
\end{proof}

Now, we give a factorization result of integral Lip-Linear operators like
that given by Villanueva for integral $m$-linear operators \cite{v}.

\begin{theorem}
Given a Lip-Linear operator $T:X\times E\rightarrow F$. Then, $T$ is
integral if and only if there exist a regular Borel measure $\mu $ on $%
B_{X^{\#}}\times B_{E^{\ast }}$ and a linear operator $b:L_{1}\left( \mu
\right) \longrightarrow F^{\ast \ast }$ such that the following diagram
commutes:%
\begin{equation}
\begin{array}{ccccc}
X\times E & \overset{T}{\longrightarrow } & F & \longrightarrow & F^{\ast
\ast } \\ 
\sigma ^{LipL}\downarrow &  &  &  & \uparrow b \\ 
X\widehat{\boxtimes }_{\varepsilon }E & \overset{\gamma }{\longrightarrow }
& \mathcal{C}\left( B_{X^{\#}}\right) \widehat{\otimes }_{\varepsilon }%
\mathcal{C}\left( B_{E^{\ast }}\right) & \overset{J}{\longrightarrow } & 
L_{1}\left( \mu \right)%
\end{array}
\label{35}
\end{equation}%
where $J$ is the natural inclusion and $\gamma $ is embedding map defined by%
\begin{equation*}
\gamma \left( \delta _{x}\boxtimes e\right) =\delta _{x}\otimes e.
\end{equation*}
\end{theorem}

\begin{proof}
Suppose $T:X\times E\rightarrow F$ is intgral. By Proposition $\ref%
{ProTandTbil}$, its corresponding bilinear $T_{bil}$ is also integral. By a
factorization result, essentially due to Villanueva \cite{v}, there exist a
regular Borel measure $\mu $ on $B_{X^{\#}}\times B_{E^{\ast }}$ a linear
operator $b:L_{1}\left( \mu \right) \longrightarrow F^{\ast \ast }$ such
that $k_{F}\circ T$ factorizes as%
\begin{equation*}
\begin{array}{ccccccc}
\mathcal{F}\left( X\right) \times E & \overset{i_{\mathcal{F}\left( X\right)
}\otimes i_{E}}{\longrightarrow } & \mathcal{C}\left( B_{X^{\#}}\right) 
\widehat{\otimes }_{\varepsilon }\mathcal{C}\left( B_{E^{\ast }}\right) & 
\overset{J}{\longrightarrow } & L_{1}\left( \mu \right) & \overset{b}{%
\longrightarrow } & F^{\ast \ast }%
\end{array}%
\end{equation*}%
In other words, $k_{F}\circ T_{bil}=b\circ J\circ \left( i_{\mathcal{F}%
\left( X\right) }\times i_{E}\right) .$ From \cite[Page 50 and 57]{ryn} and $%
\left( \ref{1,5}\right) $, we have 
\begin{equation*}
X\widehat{\boxtimes }_{\varepsilon }E=\text{$\mathcal{F}$}\left( E\right) 
\widehat{\otimes }_{\varepsilon }E\subset \mathcal{C}\left( B_{X^{\#}}\times
B_{E^{\ast }}\right) =\mathcal{C}\left( B_{X^{\#}}\right) \widehat{\otimes }%
_{\varepsilon }\mathcal{C}\left( B_{E^{\ast }}\right) .
\end{equation*}%
We define the embedding map%
\begin{equation*}
\begin{array}{ccc}
X\widehat{\boxtimes }_{\varepsilon }E & \overset{\gamma }{\longrightarrow }
& \mathcal{C}\left( B_{X^{\#}}\right) \widehat{\otimes }_{\varepsilon }%
\mathcal{C}\left( B_{E^{\ast }}\right)%
\end{array}%
\end{equation*}%
where $\gamma \circ \sigma ^{LipL}\left( x,e\right) =\delta _{x}\otimes e.$
Using this, $T$ factorizes as%
\begin{eqnarray*}
k_{F}\circ T &=&k_{F}\circ T_{bil}\circ \left( \delta _{X}\times
id_{E}\right) \\
&=&b\circ J\circ \left( i_{\mathcal{F}\left( X\right) }\times i_{E}\right)
\circ \left( \delta _{X}\times id_{E}\right) \\
&=&b\circ J\circ \gamma \circ \sigma ^{LipL}
\end{eqnarray*}%
Conversely, suppose there exists a regular Borel measure $\mu $ on $%
B_{X^{\#}}\times B_{E^{\ast }}$ and a linear operator $b:L_{1}\left( \mu
\right) \longrightarrow F^{\ast \ast }$ such that the following diagram $%
\left( \ref{35}\right) $ commutes. Then, $T_{bil}$ factorizes as%
\begin{eqnarray*}
k_{E}\circ T_{bil} &=&k_{E}\circ \widehat{T}\circ I^{-1}\circ \sigma \\
&=&\widehat{k_{E}\circ T}\circ I^{-1}\circ \sigma \\
&=&\widehat{b\circ J\circ \gamma \circ \sigma ^{LipL}}\circ I^{-1}\circ
\sigma \\
&=&b\circ J\circ \widehat{\gamma \circ \sigma ^{LipL}}\circ I^{-1}\circ
\sigma
\end{eqnarray*}%
Finally, since 
\begin{equation*}
\widehat{\gamma \circ \sigma ^{LipL}}\circ I^{-1}\circ \sigma =i_{\mathcal{F}%
\left( X\right) }\otimes i_{E}
\end{equation*}%
we obtain the desired result.
\end{proof}

\begin{proposition}
Let $T$ be a Lip-Linear form on $X\times E$. Then $T$ is integral if and
only if there exists a regular Borel measure $\mu $ on $B_{X^{\#}}\times
B_{E^{\ast }}$\textperiodcentered\ such that%
\begin{equation*}
T\left( x,e\right) =\dint\limits_{B_{X^{\#}}\times B_{E^{\ast }}}f\left(
x\right) e^{\ast }\left( e\right) d\mu \left( f,e^{\ast }\right)
\end{equation*}%
for every $x\in X$ and $e\in E.$ Furthermore, the integral norm of $T$
satisfies%
\begin{equation*}
\mathfrak{I}_{int}^{LipL}=\inf \left\{ \mu \right\} ,
\end{equation*}%
where $\mu $ is regular Borel measure on $B_{X^{\#}}\times B_{E^{\ast }}$
that satisfies the above condition.
\end{proposition}

\begin{proof}
Suppose $T$ is integral, by $\left( \ref{1,5}\right) $, $\widehat{T}$ is a
bounded on $X\widehat{\boxtimes }_{\varepsilon }E.$ By the identification $%
\left( \ref{1,5}\right) $, we deduce that $\widehat{T}\circ I^{-1}$ is a
bounded linear form on $\mathcal{F}\left( X\right) \widehat{\otimes }%
_{\varepsilon }E,$ and its corresponding bilinear form $B_{\widehat{T}\circ
I^{-1}}$ is defined through the following commutative diagram%
\begin{equation*}
\begin{array}{ccc}
\text{$\mathcal{F}$}\left( X\right) \widehat{\otimes }_{\varepsilon }E%
\overset{I^{-1}}{\longrightarrow } & X\widehat{\boxtimes }_{\varepsilon }E%
\overset{\widehat{T}}{\longrightarrow } & \mathbb{K} \\ 
\uparrow \sigma & B_{\widehat{T}\circ I^{-1}}\nearrow &  \\ 
\text{$\mathcal{F}$}\left( X\right) \times E &  & 
\end{array}%
\end{equation*}%
By \cite[Proposition 3.14]{ryn}, there exists a regular Borel measure $\mu $
on $B_{X^{\#}}\times B_{E^{\ast }}$\textperiodcentered\ such that%
\begin{equation*}
B_{\widehat{T}\circ I^{-1}}\left( m,e\right) =\dint\limits_{B_{X^{\#}}\times
B_{E^{\ast }}}f\left( m\right) e^{\ast }\left( e\right) d\mu \left(
f,e^{\ast }\right) .
\end{equation*}%
If we set $m=\delta _{\left( x,0\right) }$, we find%
\begin{equation*}
B_{\widehat{T}\circ I^{-1}}\left( \delta _{\left( x,0\right) },e\right) =%
\widehat{T}\circ I^{-1}\circ \sigma \left( \delta _{\left( x,0\right)
},e\right) =\widehat{T}\left( \delta _{\left( x,0\right) }\boxtimes e\right)
=T\left( x,e\right) ,
\end{equation*}%
which gives the desired result. Conversely, suppose there exists a regular
Borel measure $\mu $ on $B_{X^{\#}}\times B_{E^{\ast }}$\textperiodcentered\
such that%
\begin{equation*}
T\left( x,e\right) =\dint\limits_{B_{X^{\#}}\times B_{E^{\ast }}}f\left(
x\right) e^{\ast }\left( e\right) d\mu \left( f,e^{\ast }\right) .
\end{equation*}%
For $m=\sum_{i=1}^{n}\lambda _{i}\delta _{\left( x_{i},y_{i}\right) }$ and $%
e\in E,$ we have%
\begin{eqnarray*}
B_{\widehat{T}\circ I^{-1}}\left( m,e\right) &=&\widehat{T}\circ I^{-1}\circ
\sigma \left( m,e\right) =\widehat{T}\circ I^{-1}\left( m\otimes e\right) \\
&=&\widehat{T}\left( m\boxtimes e\right) =\sum_{i=1}^{n}T\left(
x_{i},\lambda _{i}e\right) -T\left( y_{i},\lambda _{i}e\right)
\end{eqnarray*}%
On the other hand,%
\begin{eqnarray*}
\sum_{i=1}^{n}T\left( x_{i},\lambda _{i}e\right) -T\left( y_{i},\lambda
_{i}e\right) &=&\sum_{i=1}^{n}\dint\limits_{B_{X^{\#}}\times B_{E^{\ast
}}}\left( f\left( x_{i}\right) -f\left( y_{i}\right) \right) e^{\ast }\left(
\lambda _{i}e\right) d\mu \left( f,e^{\ast }\right) \\
&=&\dint\limits_{B_{X^{\#}}\times B_{E^{\ast }}}f\left(
\sum_{i=1}^{n}\lambda _{i}\delta _{\left( x_{i},y_{i}\right) }\right)
e^{\ast }\left( e\right) d\mu \left( f,e^{\ast }\right) \\
&=&\dint\limits_{B_{X^{\#}}\times B_{E^{\ast }}}f\left( m\right) e^{\ast
}\left( e\right) d\mu \left( f,e^{\ast }\right) .
\end{eqnarray*}%
Therefore, 
\begin{equation*}
B_{\widehat{T}\circ I^{-1}}\left( m,e\right) =\dint\limits_{B_{X^{\#}}\times
B_{E^{\ast }}}f\left( m\right) e^{\ast }\left( e\right) d\mu \left(
f,e^{\ast }\right)
\end{equation*}%
This equality holds for elements of the form $m=\sum_{i=1}^{n}\lambda
_{i}\delta _{\left( x_{i},y_{i}\right) }$. By density, it extends to $%
\mathcal{F}\left( X\right) \times E.$ Then, by \cite[Proposition 3.14]{ryn}, 
$\widehat{T}\circ I^{-1}$ is bounded on $\mathcal{F}\left( X\right) \widehat{%
\otimes }_{\varepsilon }E$, which implies that $\widehat{T}$ is bounded on $X%
\widehat{\boxtimes }_{\varepsilon }E$.
\end{proof}

\begin{proposition}
Let $X$ be a pointed metric space and $E$ a Banach space. A Lip-Linear form $%
T$ on $X\times E$ is integral if and only if there exists a finite measure
space $\left( \Omega ,\Sigma ,\mu \right) $ and a Lipschitz operator $%
R:X\longrightarrow L_{\infty }\left( \mu \right) ,$ a linear operator, $%
v:E\longrightarrow L_{\infty }\left( \mu \right) $ such that%
\begin{equation}
T\left( x,e\right) =\dint\limits_{\Omega }R\left( x\right) v\left( e\right)
d\mu  \label{36}
\end{equation}%
for every $x\in X$ and $e\in E.$ We have%
\begin{equation*}
\mathfrak{I}_{int}^{LipL}(T)=\inf \left\Vert v\right\Vert Lip\left( R\right)
\mu \left( \Omega \right) ,
\end{equation*}%
where the infimum is taken over all such factorizations, and this infimum is
attained.
\end{proposition}

\begin{proof}
Suppose $T:X\times E\rightarrow \mathbb{K}$ is integral. Then, $\widehat{T}$
is bounded on $X\widehat{\boxtimes }_{\varepsilon }E,$ which implies that $%
\widehat{T}\circ I^{-1}$ is boundend on $\mathcal{F}\left( E\right) \widehat{%
\otimes }_{\varepsilon }F.$ By \cite[Theorem 1.3]{cab}, there exists a
finite measure space $\left( \Omega ,\Sigma ,\mu \right) $ and a Lipschitz
operator $R:X\longrightarrow L_{\infty }\left( \mu \right) ,$ a linear
operator, $v:E\longrightarrow L_{\infty }\left( \mu \right) $ such that%
\begin{equation*}
\widehat{T}\circ I^{-1}\left( \delta _{x}\otimes e\right)
=\dint\limits_{\Omega }R\left( x\right) v\left( e\right) d\mu ,\text{ for }%
x\in X\text{ and }e\in E.
\end{equation*}%
For $x\in X$ and $e\in E$%
\begin{equation*}
T\left( x,e\right) =\widehat{T}\left( \delta _{x}\boxtimes e\right) =%
\widehat{T}\circ I^{-1}\left( \delta _{x}\otimes e\right)
=\dint\limits_{\Omega }R\left( x\right) v\left( e\right) d\mu .
\end{equation*}%
Conversely, suppose there exist a Lipschitz operator $R:X\longrightarrow
L_{\infty }\left( \mu \right) $ and a linear operator, $v:E\longrightarrow
L_{\infty }\left( \mu \right) $ such that the equality $\left( \ref{36}%
\right) $ holds. Then,%
\begin{equation*}
T\left( x,e\right) =\widehat{T}\circ I^{-1}\left( \delta _{x}\otimes
e\right) =\dint\limits_{\Omega }R\left( x\right) v\left( e\right) d\mu .
\end{equation*}%
By \cite[Theorem 1.3]{cab}, $\widehat{T}\circ I^{-1}$ is bouneded on $%
\mathcal{F}\left( X\right) \widehat{\otimes }_{\varepsilon }E$, and
consequentely $\widehat{T}$ is also bouneded on $X\widehat{\boxtimes }%
_{\varepsilon }E.$ Therefore, by $\left( \ref{34}\right) ,$ $T$ is integral.
\end{proof}

\section{\textsc{Dominated (p,q)-summing Lip-Linear operators}}

The class of absolutely $(s;p,q)$-summing bilinear operators was introduced
by Matos in \cite{m}. This space is denoted by $\Pi _{s;p,q}(E\times F;G),$
with the corresponding norm $\pi _{x;p,q}\left( \cdot \right) $. When $\frac{%
1}{s}=\frac{1}{p}+\frac{1}{q}$, these operators are apply dominated $(p,q)$%
-summing. In this section, we extend this concept to Lip-Linear operators by
defining the vector space of dominated $(p,q)$-summing Lip-Linear operators,
which forms a Banach space. We establish a fundamental relationship between
Lipschitz $p$-summing operators and dominated $(p,q)$-summing Lip-Linear
operators. Specifically, a linear operator $R:X\rightarrow E$ is a Lipschitz 
$p$-summing if and only if its associated Lip-Linear form $T_{R}:X\times
E^{\ast }\rightarrow \mathbb{K}$ is dominated $(p,q)$-summing.

\begin{definition}
Let $1\leq s,p,q<\infty $ with $\frac{1}{s}=\frac{1}{p}+\frac{1}{q}.$ Let $X$
be a pointed metric space, and let $E,F$ be Banach spaces. A mapping $T\in
LipL_{0}(X\times E;F)$ is called dominated $(s;p,q)$-summing if there exists
a constant $C>0$ such that for any $x_{1},\ldots ,x_{n},y_{1},\ldots ,y_{n}$
in $X$ and $e_{1},\ldots ,e_{n}$ in $E$ we have 
\begin{equation}
\left\Vert (T(x_{i},e_{i})-T(y_{i},e_{i})))_{i=1}^{n}\right\Vert _{s}\leq
C\sup_{f\in B_{X^{\#}}}\Vert (f(x_{i})-f(y_{i}))_{i=1}^{n}\Vert _{p}\Vert
(e_{i})_{i=1}^{n}\Vert _{q,w}.  \label{41}
\end{equation}%
This class of all dominated $(p,q)$-summing Lip-Linear operators is denoted
by $\Delta _{\left( p,q\right) }^{LipL}\left( X\times E;F\right) $. In this
case, we define the norm%
\begin{equation*}
\delta _{\left( p,q\right) }^{LipL}\left( T\right) =\inf \left\{ C:C\text{
satisfies }\left( \ref{41}\right) \right\} .
\end{equation*}%
Equipped with this norm, the pair $\left( \Delta _{\left( p,q\right)
}^{LipL}\left( X\times E;F\right) ,\delta _{\left( p,q\right) }^{LipL}\left(
\cdot \right) \right) $ forms a Banach space, where we omit de proof.
\end{definition}

An important characterization of dominated $(p,q)$-summing Lip-Linear
operators is given through the Pietsch domination theorem. This allows us to
derive several interesting results. The proof relies on the generalized
Pietsch domination theorem established by Pellegrino et al. in \cite{ps} and 
\cite{pss}.

\begin{theorem}
\label{domi01} Let $1\leq s,p,q<\infty $ with $\frac{1}{s}=\frac{1}{p}+\frac{%
1}{q},$ Let $T\in LipL_{0}(X\times E;F)$. The operator $T$ is dominated $%
(p,q)$-summing if and only if there exist Radon probability measures $\mu
_{1}$ on $B_{X^{\#}}$ and $\mu _{2}$ on $B_{E^{\ast }}$ such that for all $%
x,y\in X$ and $e\in E$, we have 
\begin{equation}
\left\Vert T(x,e)-T(y,e)\right\Vert \leq C\left( \int_{B_{X^{\#}}}\left\vert
f(x)-f\left( y\right) \right\vert ^{p}d\mu _{1}\right) ^{\frac{1}{p}}\left(
\int_{B_{E^{\ast }}}\left\vert e^{\ast }(e)\right\vert ^{q}d\mu _{2}\right)
^{\frac{1}{q}}  \label{42}
\end{equation}%
Moreover, in this case $\delta _{\left( p,q\right) }^{LipL}(T)=\inf \{C:C%
\text{ verifies}\ \left( \ref{42}\right) \}$.
\end{theorem}

\begin{proof}
Choosing the parameters%
\begin{equation*}
\left\{ 
\begin{array}{l}
R_{1}:B_{X^{\#}}\times \left( X\times X\times E\right) \times \mathbb{K}%
\rightarrow \mathbb{R}^{+}:R_{1}\left( f,\left( x,y,e\right) ,\lambda
^{1}\right) =\left\vert f\left( x\right) -f\left( y\right) \right\vert \\ 
R_{2}:B_{E^{\ast }}\times \left( X\times X\times E\right) \times \mathbb{K}%
\rightarrow \mathbb{R}^{+}:R_{2}\left( e^{\ast },\left( x,y,e\right)
,\lambda ^{2}\right) =\left\vert \lambda ^{2}\right\vert \left\vert e^{\ast
}\left( e\right) \right\vert \\ 
S:LipL_{0}(X,E;F)\times \left( X\times X\times E\right) \times \mathbb{K}%
\times \mathbb{K}\rightarrow \mathbb{R}^{+}: \\ 
S\left( T,\left( x,y,e\right) ,\lambda ^{1},\lambda ^{2}\right) =\left\vert
\lambda ^{2}\right\vert \left\Vert T\left( x,e\right) -T\left( y,e\right)
\right\Vert .%
\end{array}%
\right.
\end{equation*}%
These maps satisfy conditions $\left( 1\right) $ and $\left( 2\right) $ from 
\cite[Page 1255]{pss}, allowing us to conclude that $T:X\times E\rightarrow
F $ is dominated $(p,q)$-summing if and only if 
\begin{eqnarray*}
&&(\sum_{i=1}^{n}S\left( T,\left( x_{i},y_{i},e_{i}\right) ,\lambda
_{i}^{1},\lambda _{i}^{2}\right) ^{s})^{\frac{1}{s}} \\
&\leq &\sup_{f\in B_{X^{\#}}}(\sum_{i=1}^{n}R_{1}\left( f,\left(
x_{i},y_{i},e_{i}\right) ,\lambda _{i}^{1}\right) ^{p})^{\frac{1}{p}%
}\sup_{e^{\ast }\in B_{E^{\ast }}}(\sum_{i=1}^{n}R_{2}\left( e^{\ast
},\left( x_{i},y_{i},e_{i}\right) ,\lambda _{i}^{2}\right) ^{q})^{\frac{1}{q}%
}.
\end{eqnarray*}%
Thus, $T$ is $R_{1},R_{2}$-$S$-abstract $(p,q)$-summing. A result in \cite[%
Theorem 4.6]{pss} states that $T$ is $R_{1},R_{2}$-$S$-abstract $(p,q)$%
-summing if and only if there exists a positive constant $C$ and Radon
probability measures $\mu _{1}$ on $B_{X^{\#}}$ and $\mu _{2}$ on $%
B_{E^{\ast }}$ such that 
\begin{eqnarray*}
&&S\left( T,\left( x,y,e\right) ,\lambda _{1},\lambda _{2}\right) \\
&\leq &C\left( \int_{B_{X^{\#}}}R_{1}\left( f,\left( x,y,e\right) ,\lambda
_{1}\right) ^{p}d\mu _{1}\right) ^{\frac{1}{p}}\left( \int_{B_{E^{\ast
}}}R_{2}\left( e^{\ast },\left( x,y,e\right) ,\lambda _{2}\right) ^{q}d\mu
_{2}\right) ^{\frac{1}{q}}
\end{eqnarray*}%
Consequently, 
\begin{equation*}
\left\Vert T(x,e)-T(y,e)\right\Vert \leq C\left( \int_{B_{X^{\#}}}\left\vert
f(x)-f\left( y\right) \right\vert ^{p}d\mu _{1}\right) ^{\frac{1}{p}}\left(
\int_{B_{E^{\ast }}}\left\vert e^{\ast }(e)\right\vert ^{q}d\mu _{2}\right)
^{\frac{1}{q}},
\end{equation*}%
this concludes the proof.
\end{proof}

\begin{proposition}
Let $1\leq s,p,q<\infty $ with $\frac{1}{s}=\frac{1}{p}+\frac{1}{q}.$ Let $%
X,Y$ be pointed metric spaces and let $E,G,F$ be Banach spaces. If $R\in \Pi
_{p}^{L}(X;Y)$, $v\in \Pi _{q}(E;G)$ and $S\in LipL_{0}(Y\times G;F)$, then $%
S\circ (R,v)$ is dominated $(p,q)$-summing and%
\begin{equation*}
\delta _{\left( p,q\right) }^{LipL}(S\circ (R,v))\leq LipL(S)\pi
_{p}^{L}(R)\pi _{q}(v).
\end{equation*}
\end{proposition}

\begin{proof}
For all $x,y\in X$ and $e\in E$ we have 
\begin{equation*}
\left\Vert S(R(x),v(e))-S(R(y),v(e))\right\Vert \leq
LipL(S)d(R(x),R(y))\left\Vert v(e)\right\Vert .
\end{equation*}%
Since $R$ is Lipschitz $p$-summing, by \cite[Theorem 1]{Farmer},\ there
exists a Radon probability measure $\mu _{1}$ on $B_{X^{\#}}$ such that 
\begin{equation*}
d(R(x),R(y))\leq \pi _{p}^{L}\left( R\right) \left(
\int_{B_{X^{\#}}}\left\vert f(x)-f\left( y\right) \right\vert ^{p}d\mu
_{1}\right) ^{\frac{1}{p}}.
\end{equation*}%
Similarly, as $v$ is $q$-summing, by \cite[Theorem 2.12]{distel}, there
exists a Radon probability measure $\mu _{2}$ on $B_{E^{\ast }}$ such that 
\begin{equation*}
\left\Vert v(e)\right\Vert \leq \pi _{q}\left( v\right) \left(
\int_{B_{E^{\ast }}}\left\vert e^{\ast }(e)\right\vert ^{q}d\mu _{2}\right)
^{\frac{1}{q}}.
\end{equation*}%
Therefore, by Theorem \ref{domi01}$,$ $T$ is dominated $(p,q)$-summing and 
\begin{equation*}
\delta _{\left( p,q\right) }^{LipL}(T)\leq LipL(S)\pi _{p}^{L}(R)\pi _{q}(v).
\end{equation*}
\end{proof}

\begin{theorem}
\label{equiTheorem}Let $1\leq s,p,q<\infty $ with $\frac{1}{s}=\frac{1}{p}+%
\frac{1}{q}.$ For $T\in LipL_{0}(X\times E;F)$, the following assertions are
equivalent.

1) The Lip-Linear operator $T$ belongs to $\Delta _{\left( p,q\right)
}^{LipL}(X\times E;F)$.

2) The Lipschitz operator $A_{T}$ belongs to $\Pi _{p}^{L}(X;\Pi _{q}(E;F))$.

3) The linear operator $B_{T}$ belongs to $\Pi _{q}(E;\Pi _{p}^{L}(X;F))$.

Additionally, the following equality holds%
\begin{equation*}
\delta _{\left( p,q\right) }^{LipL}(T)=\pi _{p}^{L}(A_{T})=\pi _{q}(B_{T}).
\end{equation*}%
As a direct consequence, we have the following isometric identifications%
\begin{equation}
\Delta _{\left( p,q\right) }^{LipL}(X\times E;F)=\Pi _{p}^{L}(X;\Pi
_{q}(E;F))=\Pi _{q}(E;\Pi _{p}^{L}(X;F)).  \label{43}
\end{equation}
\end{theorem}

\begin{proof}
$1)\Longrightarrow 2):$ Let $T\in \Delta _{\left( p,q\right)
}^{LipL}(X\times E;F)$ and fix $x,y\in X.$ We aim to show that $%
A_{T}(x)-A_{T}\left( y\right) \in \Pi _{q}(E;F).$ Let $e\in E$. By Theorem %
\ref{domi01}, 
\begin{eqnarray*}
\left\Vert \left( A_{T}(x)-A_{T}\left( y\right) \right) \left( e\right)
\right\Vert &=&\left\Vert A_{T}(x)\left( e\right) -A_{T}(y)\left( e\right)
\right\Vert =\left\Vert T(x,e)-T(y,e)\right\Vert \\
&\leq &\delta _{\left( p,q\right) }^{LipL}(T)(\int_{B_{X^{\#}}}\left\vert
f(x)-f\left( y\right) \right\vert ^{p}d\mu _{1})^{\frac{1}{p}%
}(\int_{B_{E^{\ast }}}\left\vert e^{\ast }(e)\right\vert ^{q}d\mu _{2})^{%
\frac{1}{q}} \\
&\leq &\delta _{\left( p,q\right) }^{LipL}(T)(\int_{B_{X^{\#}}}\left\vert
f(x)-f\left( y\right) \right\vert ^{p}d\mu _{1})^{\frac{1}{p}%
}(\int_{B_{E^{\ast }}}\left\vert e^{\ast }(e)\right\vert ^{q}d\mu _{2})^{%
\frac{1}{q}}.
\end{eqnarray*}%
This shows that $A_{T}(x)-A_{T}\left( y\right) \in \Pi _{q}(E;F)$ with%
\begin{equation}
\pi _{q}\left( A_{T}(x)-A_{T}\left( y\right) \right) \leq \delta _{\left(
p,q\right) }^{LipL}(T)\left( \int_{B_{X^{\#}}}\left\vert f(x)-f\left(
y\right) \right\vert ^{p}d\mu _{1}\right) ^{\frac{1}{p}}.  \label{44}
\end{equation}%
Setting $y=0$ in $\left( \ref{44}\right) $ implies $A_{T}(x)\in \Pi
_{q}(E;F).$ Therefore, $A_{T}$ is Lipschitz $p$-summing, and%
\begin{equation*}
\pi _{p}^{L}(A_{T})\leq \delta _{\left( p,q\right) }^{LipL}(T).
\end{equation*}%
$2)\Longrightarrow 3):$ Let $e\in E.$ We aim to show $B_{T}\left( e\right)
\in \Pi _{p}^{L}(X;F).$ For $x,y\in X,$%
\begin{equation*}
\left\Vert B_{T}\left( e\right) \left( x\right) -B_{T}\left( e\right) \left(
y\right) \right\Vert =\left\Vert A_{T}(x)\left( e\right) -A_{T}(y)\left(
e\right) \right\Vert .
\end{equation*}%
Using the $q$-summing property of $A_{T}(x)-A_{T}\left( y\right) ,$ followed
the Lipschitz $p$-summing property of $A_{T},$ we obtain%
\begin{eqnarray*}
\left\Vert B_{T}\left( e\right) \left( x\right) -B_{T}\left( e\right) \left(
y\right) \right\Vert &\leq &\pi _{q}\left( A_{T}(x)-A_{T}\left( y\right)
\right) \left( \int_{B_{E^{\ast }}}\left\vert e^{\ast }(e)\right\vert
^{q}d\mu _{2}\right) ^{\frac{1}{q}} \\
&\leq &\pi _{p}^{L}(A_{T})\left( \int_{B_{X^{\#}}}\left\vert f(x)-f\left(
y\right) \right\vert ^{p}d\mu _{1}\right) ^{\frac{1}{p}}\left(
\int_{B_{E^{\ast }}}\left\vert e^{\ast }(e)\right\vert ^{q}d\mu _{2}\right)
^{\frac{1}{q}}.
\end{eqnarray*}%
This shows that $B_{T}\left( e\right) $ is Lipschitz $p$-summing, and%
\begin{equation*}
\pi _{p}^{L}\left( B_{T}\left( e\right) \right) \leq \pi
_{p}^{L}(A_{T})\left( \int_{B_{E^{\ast }}}\left\vert e^{\ast }(e)\right\vert
^{q}d\mu _{2}\right) ^{\frac{1}{q}}.
\end{equation*}%
Thus, $B_{T}$ is $q$-summing, and%
\begin{equation*}
\pi _{q}\left( B_{T}\right) \leq \pi _{p}^{L}(A_{T}).
\end{equation*}%
$3)\Longrightarrow 1):$ Let $x,y\in X$ and $e\in E$. Then,%
\begin{equation*}
\Vert T(x,e)-T(y,e)\Vert =\left\Vert B_{T}\left( e\right) \left( x\right)
-B_{T}\left( e\right) \left( y\right) \right\Vert .
\end{equation*}%
Since $B_{T}\left( e\right) $ is Lipschitz $p$-summing, there exists a Radon
probability measure $\mu _{1}$ on $B_{X^{\#}}$ such that%
\begin{eqnarray*}
\Vert T(x,e)-T(y,e)\Vert &=&\left\Vert B_{T}\left( e\right) \left( x\right)
-B_{T}\left( e\right) \left( y\right) \right\Vert \\
&\leq &\pi _{p}^{L}(B_{T}\left( e\right) )\left( \int_{B_{X^{\#}}}\left\vert
f(x)-f\left( y\right) \right\vert ^{p}d\mu _{1}\right) ^{\frac{1}{p}}.
\end{eqnarray*}%
Given that $B_{T}$ is $q$-summing, there exists a Radon probability measures 
$\mu _{2}$ on $B_{E^{\ast }}$ such that%
\begin{equation*}
\Vert T(x,e)-T(y,e)\Vert \leq \pi _{q}(B_{T})\left(
\int_{B_{X^{\#}}}\left\vert f(x)-f\left( y\right) \right\vert ^{p}d\mu
_{1}\right) ^{\frac{1}{p}}\left( \int_{B_{E^{\ast }}}\left\vert e^{\ast
}(e)\right\vert ^{q}d\mu _{2}\right) ^{\frac{1}{q}}.
\end{equation*}%
Thus, $T$ is dominated $(p,q)$-summing, and%
\begin{equation*}
\delta _{\left( p,q\right) }^{LipL}(T)\leq \pi _{q}(B_{T}).
\end{equation*}%
To complete the proof, we need to establish the surjectivity. Let $S\in \Pi
_{p}^{L}(X;\Pi _{q}(E;F))$. Define a Lip-Linear operator $T_{S}:X\times
E\rightarrow F$ by%
\begin{equation*}
T_{S}\left( x,e\right) =S\left( x\right) \left( e\right) .
\end{equation*}%
We claim that $T_{S}$ is dominated $(p,q)$-summing. Indeed, let $x,y\in X$
and $e\in E$. Then,%
\begin{equation*}
\left\Vert T_{S}(x,e)-T_{S}(y,e)\right\Vert =\left\Vert S\left( x\right)
\left( e\right) -S\left( y\right) \left( e\right) \right\Vert .
\end{equation*}%
Using the $q$-summing property of $S\left( x\right) -S\left( y\right) ,$%
\begin{equation*}
\left\Vert T_{S}(x,e)-T_{S}(y,e)\right\Vert \leq \pi _{q}\left( S\left(
x\right) -S\left( y\right) \right) \left( \int_{B_{E^{\ast }}}\left\vert
e^{\ast }(e)\right\vert ^{q}d\mu _{2}\right) ^{\frac{1}{q}}
\end{equation*}%
Since $S\in \Pi _{p}^{L}(X;\Pi _{q}(E;F)),$ we have%
\begin{equation*}
\left\Vert T_{S}(x,e)-T_{S}(y,e)\right\Vert \leq \pi _{p}^{L}\left( S\right)
\left( \int_{B_{X^{\#}}}\left\vert f(x)-f\left( y\right) \right\vert
^{p}d\mu _{1}\right) ^{\frac{1}{p}}\left( \int_{B_{E^{\ast }}}\left\vert
e^{\ast }(e)\right\vert ^{q}d\mu _{2}\right) ^{\frac{1}{q}}.
\end{equation*}%
This proves that $T_{S}$ dominated $(p,q)$-summing, and hence $S$ induces a
surjective correspondence. By similar arguments, a corresponding
identification holds for an operator in $\Pi _{q}(E;\Pi _{p}^{L}(X;F)).$
Thus, the isometric identification $\left( \ref{43}\right) $ is established.
\end{proof}

\begin{theorem}
Let $1\leq s,p,q<\infty $ with $\frac{1}{s}=\frac{1}{p}+\frac{1}{q}.$ The
following assertions are equivalent.

1) The Lipschitz operator $R:X\rightarrow E$ is Lipschitz $p$-summing.

2) The Lip-Linear operator $T_{R}$ belongs to $\Delta _{\left( p,q\right)
}^{LipL}(X\times E^{\ast })$.

Additionally, the following equality holds%
\begin{equation*}
\pi _{p}^{L}\left( R\right) =\delta _{\left( p,q\right) }^{LipL}(T).
\end{equation*}%
As a direct consequence, we have the following isometric identification%
\begin{equation*}
\Pi _{p}^{L}(X;E)=\Delta _{\left( p,q\right) }^{LipL}(X\times E^{\ast }).
\end{equation*}
\end{theorem}

\begin{proof}
It is straightforward to see that $k_{E}\circ R=A_{T_{R}},$%
\begin{equation*}
\begin{array}{ccc}
X & \overset{R}{\longrightarrow } & E \\ 
& \searrow A_{T_{R}} & \downarrow k_{E} \\ 
&  & E^{\ast \ast }%
\end{array}%
\end{equation*}%
Indeed, for $x\in X$ and $e^{\ast }\in E^{\ast }$ we have%
\begin{equation*}
k_{F}\circ R\left( x\right) \left( e^{\ast }\right) =\left\langle e^{\ast
},R\left( x\right) \right\rangle =T_{R}\left( x,e^{\ast }\right)
=A_{T_{R}}\left( x\right) \left( e^{\ast }\right)
\end{equation*}%
Consequently, by the injectivity property of Lipschitz $p$-summing
operators, $R$ is Lipschitz $p$-summing if and only if $A_{T_{R}}$ is
Lipschitz $p$-summing and%
\begin{equation*}
\pi _{p}^{L}\left( R\right) =\pi _{p}^{L}\left( k_{F}\circ R\right) =\pi
_{p}^{L}\left( A_{T_{R}}\right) =\delta _{\left( p,q\right) }^{LipL}(T).
\end{equation*}
\end{proof}

If $E,F$ and $G$ are Banach spaces and $T:E\times F\rightarrow G$ is a
bilinear operator, the associated operators 
\begin{equation*}
\begin{array}{llll}
A_{T}: & E & \longrightarrow & \mathcal{L}\left( F;G\right) \\ 
& x & \mapsto & A_{T}\left( x\right) \left( y\right) =T\left( x,y\right)%
\end{array}%
\text{ and }%
\begin{array}{llll}
B_{T}: & F & \longrightarrow & \mathcal{L}(E;G) \\ 
& y & \mapsto & B_{T}\left( y\right) \left( x\right) =T\left( x,y\right)%
\end{array}%
\end{equation*}%
are linear. We now state the following result.

\begin{corollary}
Let $1\leq s,p,q<\infty $ with $\frac{1}{s}=\frac{1}{p}+\frac{1}{q}.$ Let $%
E,F$ and $G$ be Banach spaces. For a bilinear operator $T:E\times
F\rightarrow G$, the following assertions are equivalent

1) The bilinear operator $T$ belongs to $\Pi _{s;p,q}(E\times F;G)$.

2) The bilinear operator $T$ belongs to $\Delta _{\left( p,q\right)
}^{LipL}(E\times F;G)$.

3) The linear operator $A_{T}$ belongs to $\Pi _{p}(E;\Pi _{q}(F;G))$.

4) The linear operator $B_{T}$ belongs to $\Pi _{q}(F;\Pi _{p}(E;G)).$

Moreover, we have%
\begin{equation*}
\delta _{\left( p,q\right) }^{LipL}(T)=\pi _{s;p,q}(T)=\pi _{p}(A_{T})=\pi
_{q}(B_{T}).
\end{equation*}
\end{corollary}

\begin{proof}
$1)\Longrightarrow 2):$ Assume $T\in \Pi _{s;p,q}(E\times F;G).$ For $\left(
x_{i}\right) _{i=1}^{n},\left( y_{i}\right) _{i=1}^{n}\subset E$ and $\left(
e_{i}\right) _{i=1}^{n}\in F$ we have%
\begin{eqnarray*}
\left\Vert (T(x_{i},e_{i})-T(y_{i},e_{i})))_{i=1}^{n}\right\Vert _{s}
&=&\left\Vert (T(x_{i}-y_{i},e_{i}))_{i=1}^{n}\right\Vert _{s} \\
&\leq &\pi _{s;p,q}(T)\sup_{x^{\ast }\in B_{E^{\ast }}}\Vert (x^{\ast
}(x_{i}-y_{i}))_{i=1}^{n}\Vert _{p}\Vert (e_{i})_{i=1}^{n}\Vert _{q,w} \\
&\leq &\pi _{s;p,q}(T)\sup_{x^{\ast }\in B_{E^{\ast }}}\Vert (x^{\ast
}\left( x_{i}\right) -x^{\ast }\left( y_{i}\right) )_{i=1}^{n}\Vert
_{p}\Vert (e_{i})_{i=1}^{n}\Vert _{q,w} \\
&\leq &\pi _{s;p,q}(T)\sup_{f\in B_{E^{\#}}}\Vert (f\left( x_{i}\right)
-f\left( y_{i}\right) )_{i=1}^{n}\Vert _{p}\Vert (e_{i})_{i=1}^{n}\Vert
_{q,w}.
\end{eqnarray*}%
Then, $T$ is dominated $\left( p,q\right) $-summing and 
\begin{equation*}
\delta _{\left( p,q\right) }^{LipL}(T)\leq \pi _{s;p,q}(T).
\end{equation*}%
$2)\Longrightarrow 3):$ By Theorem $\ref{equiTheorem},$ $A_{T}:E\rightarrow $
$\Pi _{q}(F;G)$ is Lipschitz $p$-summing. Using the fact, given by Farmer 
\cite{Farmer}, that the definitions of $p$-summing and Lipschitz $p$-summing
coincide for linear operators. It follows that $A_{T}$ is $p$-summing, and%
\begin{equation*}
\delta _{\left( p,q\right) }^{LipL}(T)=\pi _{p}^{L}(A_{T})=\pi _{p}(A_{T}).
\end{equation*}%
$3)\Longrightarrow 4):$ Suppose $A_{T}:E\rightarrow $ $\Pi _{q}(F;G)$ is $p$%
-summing. Then, $A_{T}$ is Lipschitz $p$-summing, By Theorem $\ref%
{equiTheorem},$ $B_{T}:F\rightarrow \Pi _{p}^{L}(E;F)$ is $q$-summing. It
suffices to show that $B_{T}\left( e\right) \in \Pi _{p}(E;F)$ for every $%
e\in F.$ Indeed, since $T$ is bilinear then $B_{T}\left( e\right) $ is
linear. Thus, we deduce that $B_{T}\left( e\right) $ is $p$-summing, and%
\begin{equation*}
\pi _{p}(A_{T})=\pi _{p}^{L}(A_{T})=\pi _{q}(B_{T}).
\end{equation*}

$4)\Longrightarrow 1):$ Let $x,y\in E$ and $e\in F.$ Since $B_{T}\left(
e\right) $ is $p$-summing, 
\begin{equation*}
\left\Vert T(x,e)\right\Vert =\left\Vert B_{T}\left( e\right) (x)\right\Vert
\leq \pi _{p}\left( B_{T}\left( e\right) \right) \left( \int_{B_{E^{\ast
}}}\left\vert x^{\ast }(x)\right\vert ^{p}d\mu \right) ^{\frac{1}{p}}
\end{equation*}%
and as $B_{T}$ is $q$-summing, we obtain%
\begin{equation*}
\left\Vert T(x,e)\right\Vert \leq \pi _{q}\left( B_{T}\right) \left(
\int_{B_{E^{\ast }}}\left\vert x^{\ast }(x)\right\vert ^{p}d\mu \right) ^{%
\frac{1}{p}}\left( \int_{B_{F^{\ast }}}\left\vert e^{\ast }(e)\right\vert
^{q}d\eta \right) ^{\frac{1}{q}}
\end{equation*}%
Then, $T$ is dominated $(p,q)$-summing, and%
\begin{equation*}
\pi _{s;p,q}(T)\leq \pi _{q}(B_{T}).
\end{equation*}
\end{proof}

\begin{proposition}[Dvoretzky-Rogers Theorem]
$\Delta _{\left( p,q\right) }^{LipL}(E\times E;E)=LipL_{0}(E\times E;E)$ if
and only if $E$ is a finite-dimensional space.
\end{proposition}

\begin{proof}
Let $e_{0}\in E$ with $\left\Vert e_{0}\right\Vert =1,$ and let $e_{0}^{\ast
}\in B_{E^{\ast }}$ satisfy $\left\langle e_{0}^{\ast },e_{0}\right\rangle
=\left\Vert e_{0}\right\Vert $. Define $T_{e_{0}^{\ast }}:E\times
E\rightarrow E$ by 
\begin{equation*}
T(e_{1},e_{2})=\left\langle e_{0}^{\ast },e_{1}\right\rangle e_{2}.
\end{equation*}%
It follows that%
\begin{equation*}
T\in \mathcal{B}(E\times E;E)\subset LipL_{0}(E\times E;E)=\Delta _{\left(
p,q\right) }^{LipL}(E\times E;E).
\end{equation*}%
For every $e\in E$, we have 
\begin{equation*}
A_{T}\left( e_{0}\right) \left( e\right) =T(e_{0},e)=\left\langle
e_{0}^{\ast },e_{0}\right\rangle e=id_{E}\left( e\right) ,
\end{equation*}%
By Theorem $\ref{equiTheorem}$, $A_{T}\left( e_{0}\right) $ is $q$-summing,
which implies $id_{E}\in \Pi _{q}\left( E;E\right) .$ By the
Dvoretzky-Rogers Theorem, this implies that $E$ is a finite-dimensional.
Conversely, suppose that $T\in LipL_{0}(E\times E;E)$ and $E$ is
finite-dimensional. Then, $A_{T}:E\rightarrow \mathcal{L}(E;E)$ is Lipschitz 
$p$-summing and for every $e\in E,$ $A_{T}\left( e\right) :E\rightarrow E$
is $q$-summing. According to Theorem $\ref{equiTheorem}$, $T$ is dominated $%
\left( p,q\right) $-summing.
\end{proof}

\end{document}